\documentclass[a4paper,11pt,french,english]{amsart}
\usepackage[english]{babel}
\usepackage[utf8]{inputenc}
\usepackage{lmodern}
\usepackage[T1]{fontenc}
\usepackage{amsmath,amsthm,amssymb,amsfonts}
\usepackage{amscd}
\usepackage{pstricks}
\usepackage{pst-node}
\usepackage[a4paper,left=3.4cm,right=3.4cm,top=4cm,bottom=4cm]{geometry}
\usepackage{enumerate}
\usepackage{enumitem}
\usepackage[linktocpage=true]{hyperref}
\usepackage{url}
\usepackage{csquotes}
\usepackage{tikz-cd}

\newcommand{\DL}{\mathrm{DL}}
\newcommand{\focal}{\mathbb{G}}
\newcommand{\elliptic}{\mathbb{E}}

\newcommand{\Aut}{\mathrm{Aut}}
\newcommand{\Comm}{\mathrm{Comm}}
\newcommand{\Isom}{\mathrm{Isom}}

\newcommand{\Z}{\mathbb{Z}}
\newcommand{\Q}{\mathbb{Q}}
\newcommand{\R}{\mathbb{R}}
\newcommand{\A}{\mathbb{A}}
\newcommand{\F}{\mathbb{F}}

\newcommand{\RadLE}{\mathrm{Rad_{LE}}}
\newcommand{\RadLF}{\mathrm{Rad_{LF}}}
\newcommand{\FC}{\mathrm{FC}}
\newcommand{\Fit}{\mathrm{Fit}}

\title[]{Rigidity and flexibility results for groups with a common cocompact envelope}

\author{Adrien Le Boudec}
\address{CNRS, UMPA - ENS Lyon, 46 all\'ee d'Italie, 69364 Lyon, France}
\email{adrien.le-boudec@ens-lyon.fr}

\thanks{}

\date{November 25, 2026}
\theoremstyle{plain}
\newtheorem{Theorem}{Theorem}[section]
\newtheorem{Proposition}[Theorem]{Proposition}
\newtheorem{Corollary}[Theorem]{Corollary}
\newtheorem{Lemma}[Theorem]{Lemma}

\newtheorem{Theorem-intro}{Theorem}
\newtheorem{Corollary-intro}[Theorem-intro]{Corollary}

\theoremstyle{definition}
\newtheorem{Definition}[Theorem]{Definition}
\newtheorem{Notation}[Theorem]{Notation}

\newtheorem{Example}[Theorem]{Example}
\newtheorem{Remark}[Theorem]{Remark}

\newtheorem{Definition-intro}{Definition}

\begin{document}

\maketitle

\begin{abstract}
	A locally compact group $G$ is a cocompact envelope of a group $\Gamma$ if $G$ contains a copy of $\Gamma$ as a discrete and cocompact subgroup. We study the problem that takes two finitely generated groups $\Gamma,\Lambda$ having a common cocompact envelope, and asks what properties must be shared between $\Gamma$ and $\Lambda$.
	
	We first consider the setting where the common cocompact envelope is totally disconnected. In that situation we show that if $\Gamma$ admits a finitely generated nilpotent normal subgroup $A$, then virtually $\Lambda$ admits a normal subgroup $B$ such that $A$ and $B$ are virtually isomorphic.
		 
	We establish both rigidity and flexibility results when $\Gamma$ belongs to the class of solvable groups of finite rank. On the rigidity perspective, we show that if $\Gamma$ is solvable of finite rank, and the locally finite radical of $\Lambda$ is finite, then $\Lambda$ must be virtually solvable of finite rank. On the flexibility perspective, we exhibit groups $\Gamma,\Lambda$ with a common cocompact envelope such that $\Gamma$ is solvable of finite rank, while $\Lambda$ is not virtually solvable. In particular the class of  solvable groups of finite rank is not QI-rigid. We also exhibit flexibility behaviours among  finitely presented groups, and more generally among groups with type $F_n$ for arbitrary $n \geq 1$. 
\end{abstract}

\section*{Introduction}

Let $G$ be a locally compact group. A subgroup $\Gamma$ of $G$ is a lattice if $\Gamma$ is discrete in $G$ and $G / \Gamma$ admits a $G$-invariant finite measure. The study of interactions between properties of $G$ and properties of its lattices is a rich topic, involving geometric, analytic and ergodic aspects. The ambient group $G$ being fixed, the problem of classifying all lattices of $G$ has been intensively studied; historically first in the prominent case of connected Lie groups and algebraic groups over local fields, and gradually for some more locally compact groups. In this paper we are concerned with a different problem, which takes as input a group $\Gamma$ and aims at studying the discrete groups $\Lambda$ such that $\Gamma$ and $\Lambda$ sit as lattices in a common locally compact group.

We fix some terminology. Let $\Gamma$ be a discrete group. A locally compact group $G$ is called an envelope of $\Gamma$ if $G$ contains a lattice isomorphic to $\Gamma$ \cite{Furstenberg-envelopes}. When the given lattice is cocompact, we say that $G$ is a cocompact envelope of $\Gamma$. We will mainly focus on the case of cocompact envelopes. We say that two discrete groups $\Gamma$ and $\Lambda$ share a cocompact envelope if there is a locally compact group $G$ such that $G$ is a common cocompact envelope of $\Gamma$ and $\Lambda$. 

Recall that an action of a group on a proper metric space by isometries is called geometric if the action is proper and cocompact. For finitely generated groups $\Gamma$ and $\Lambda$, the existence of a common cocompact envelope is equivalent to the existence of a proper metric space on which $\Gamma$ and $\Lambda$ both act faithfully and geometrically (see e.g. \cite[4.C-5.B]{CorHar}). When this holds, the Milnor-Schwarz lemma asserts that the groups $\Gamma$ and $\Lambda$ are quasi-isometric (QI hereafter).

Two groups are called virtually isomorphic if they have finite index subgroups that are isomorphic, and virtually isomorphic up to finite kernel if they have finite index subgroups that are isomorphic after modding out by a finite normal subgroup. 

\begin{Definition-intro}
	Let $\mathcal{C}$ be a class of finitely generated groups. We say that $\mathcal{C}$ is rigid for cocompact envelopes if for every group $\Lambda$ that shares a cocompact envelope with a group $\Gamma$ in $\mathcal{C}$, the group $\Lambda$ is virtually isomorphic up to finite kernel to a group in $\mathcal{C}$. We write \enquote{CE-rigid} for \enquote{rigid for cocompact envelopes}.
\end{Definition-intro}

Recall a class $\mathcal{C}$ is QI-rigid if for every group $\Lambda$ that is QI to a group $\Gamma$ in $\mathcal{C}$, the group $\Lambda$ is virtually isomorphic up to finite kernel to a group in $\mathcal{C}$. If $\mathcal{C}$ is QI-rigid then $\mathcal{C}$ is CE-rigid. Hence the question of CE-rigidity of a given class is relevant either when QI-rigidity is not known, or when QI-rigidity is known to fail. More targeted towards locally compact groups, there is the problem of \enquote{describing} all cocompact envelopes of a group $\Gamma$ in $\mathcal{C}$. This problem and the problem of CE-rigidity are both sub-cases of the more general problem of studying all locally compact groups that are commable to a group $\Gamma$  in $\mathcal{C}$ (see \S \ref{subsec-stable-commable} for the definition of two groups being commable). 

\subsection*{Totally disconnected common cocompact envelope}

We first focus on the situation where the common cocompact envelope $G$ of two groups $\Gamma$ and $\Lambda$ is a totally disconnected locally compact group. For finitely generated groups $\Gamma$ and $\Lambda$, the existence of a common totally disconnected  cocompact envelope is equivalent to the existence of a connected locally finite graph on which $\Gamma$ and $\Lambda$  act faithfully and geometrically. Indeed, if $X$ is such a graph, then the group of automorphisms of $X$ is a common totally disconnected cocompact envelope. And conversely if $G$ is such an envelope, then $\Gamma$ and $\Lambda$ act geometrically on any Cayley-Abels graph of $G$ (see e.g.\ \cite[2.E.9]{CorHar} for Cayley-Abels graphs), and $\Gamma$ and $\Lambda$ act faithfully on such a graph provided the associated compact open subgroup of $G$ is sufficiently small. See in addition the beginning of Section \ref{sec-general-result-tdlc} for the relevance of considering totally disconnected envelopes. 

Our first result deals with the class of groups $\Gamma$ having a finitely generated nilpotent normal subgroup. We prove:

\begin{Theorem-intro} \label{thm-intro-norma-nilp-same-tdlc-env}
	Let	$\Gamma$ be a finitely generated group with a normal subgroup $A \lhd \Gamma$ such that $A$ is finitely generated and nilpotent. Suppose that $\Gamma$ and $\Lambda$ share a totally disconnected cocompact envelope. Then there is a finite index subgroup $\Lambda'$ of $\Lambda$  such that $\Lambda'$ admits a normal subgroup $B \lhd \Lambda'$ such that $A$ and $B$ are virtually isomorphic. 
\end{Theorem-intro}

 It is worth comparing Theorem \ref{thm-intro-norma-nilp-same-tdlc-env} with the situation where $\Gamma$ and $\Lambda$ share a cocompact envelope, but this envelope is not necessarily totally disconnected. Equivalently, the situation where we have two groups $\Gamma$ and $\Lambda$ and a proper metric space on which $\Gamma$ and $\Lambda$ both act faithfully and geometrically, but the metric space acted upon is not necessarily a graph. The groups studied by Leary-Minasyan in \cite{Leary-Minasyan} include examples of groups $\Gamma$ and $\Lambda$ acting faithfully and geometrically on $\R^d \times T$ - the product of the $d$-dimensional Euclidean space $\R^d$ and a locally finite tree $T$ - such that $\Gamma$ is of the form $\Gamma = A \times F$ with $A = \Z^d$ and $F$ is a non-abelian free group of finite rank, and such that $\Lambda$ does not virtually admit any non-trivial abelian (or nilpotent) normal subgroup. Here the common cocompact envelope is $\mathrm{Isom}(\R^d) \times \Aut(T)$. Hence these examples show that Theorem \ref{thm-intro-norma-nilp-same-tdlc-env} does not hold when the common envelope is not totally disconnected.
 
 It is also interesting to connect the setting of Theorem \ref{thm-intro-norma-nilp-same-tdlc-env} with other rigidity results for groups with a \textit{commensurated} finitely generated nilpotent subgroup. We refer the reader to the discussion in \S  \ref{subsec-related-rigidity-results}.

A main tool in the proof of Theorem \ref{thm-intro-norma-nilp-same-tdlc-env} is Theorem \ref{thm-abelien-norm-co} below. The discussion of this statement, as well as the other consequences that we derive from it, are postponed for the moment. 

\subsection*{Solvable groups of finite rank}

We turn our attention to CE-rigidity for solvable groups. As observed by Erschler, the entire class of solvable groups is not CE-rigid. If $F_1, F_2$ are two finite groups of the same cardinality, the wreath products $\Gamma = F_1 \wr \Z$ and $\Lambda = F_2 \wr \Z$ admit a common Cayley graph, and hence a common cocompact envelope (the automorphism group of this Cayley graph). And $\Gamma$ is solvable provided $F_1$ is, while $\Lambda$ is not virtually solvable provided $F_2$ is not solvable \cite{Erschler-instab}. (Recall that being virtually isomorphic up to finite kernel to a solvable group is the same as being virtually solvable.)

Here we focus on the class of solvable groups of finite rank. This class is much more tame than the entire class of solvable groups. Recall that a solvable group $\Gamma$ is of finite rank if there is an integer $k$ such that every finitely generated subgroup of $\Gamma$ can be generated by at most $k$ elements. A wreath product $A \wr B$ is never of finite rank provided $A$ is non-trivial and $B$ is infinite. Every polycyclic group is of finite rank. The solvable Baumslag--Solitar group $\Z[1/n] \rtimes_n \Z$, $n \geq 2$, is non-polycyclic but is of finite rank. More generally, every finitely generated solvable group that is linear over $\Q$ is of finite rank. The finitely generated solvable groups linear over $\Q$ are precisely the finitely generated solvable groups of finite rank that are virtually torsion-free.

The large-scale geometry and QI-rigidity of certain families of non-polycyclic solvable groups of finite rank have been studied  by Farb--Mosher in \cite{Far-Mosh-BSI,Far-Mosh-BSII,Farb-Mosher-abelian-by-cyclic}. In the case of $\Gamma = \Z[1/n] \rtimes_n \Z$, the main result of \cite{Far-Mosh-BSII} establishes the strongest possible form of QI-rigidity:  if $\Lambda$ is a group QI to $\Gamma $, then  $\Lambda$ is virtually isomorphic up to finite kernel to $\Gamma $.

We establish both positive and negative results about CE-rigidity for the class of solvable groups of finite rank. In the direction of rigidity, we have the following:

\begin{Theorem-intro} \label{thm-intro-Rad-Lambda-fin}
	Let	$\Gamma$ be a finitely generated solvable group of finite rank. Let $\Lambda$ such that $\Gamma$ and $\Lambda$ share a cocompact envelope. Suppose that  $\Lambda$ has no normal subgroup that is infinite and locally finite. Then $\Lambda$ is  virtually solvable of finite rank. 
\end{Theorem-intro}

On the flexibility side, we provide two constructions that show CE-rigidity fails for the class of solvable groups of finite rank. These show that the assumption in Theorem  \ref{thm-intro-Rad-Lambda-fin} that $\Lambda$  has no normal subgroup that is infinite and locally finite is truly needed. We refer the reader to the end of this introduction for an outline of these constructions. The first one yields:

\begin{Theorem-intro} \label{thm-intro-not-QI-rigid}
	There exist finitely generated groups $\Gamma, \Lambda$ that share a cocompact envelope, with $\Gamma$  solvable of finite rank, and $\Lambda$  not virtually solvable.
\end{Theorem-intro}

We exhibit examples of $\Gamma$ and $\Lambda$ as in the theorem that are rather small. The group $\Gamma$ can be taken of the form $\Gamma = \Z[1/n]^2 \rtimes \Z$ for $n \geq 2$ (so $\Gamma$ is abelian-by-$\Z$ and of Hirsch number $3$), and $\Lambda$ can be taken to be $\Z^2 \times F \wr \Z$ for arbitrary finite group $F$ of cardinality $n$ (in particular non virtual solvability of $\Lambda$ arises similarly as in \cite{Erschler-instab}).

Our examples of groups $\Gamma$ and $\Lambda$ as in Theorem \ref{thm-intro-not-QI-rigid} are not finitely presented. However our second construction that shows CE-rigidity fails for the class of solvable groups of finite rank includes finitely presented groups, and even groups with higher finiteness properties. Recall that a group $\Gamma$ has type $F_n$, $n \geq 1$, if there exists a CW-complex with finite $n$-skeleton, with fundamental group $\Gamma$ and contractible universal cover \cite{Wall65}. $F_{n+1}$ implies $F_n$, and $F_1$ and $F_2$ are respectively equivalent to being finitely generated and being finitely presented. Having type $F_n$ is a QI-invariant \cite[Theorem 9.56]{Drutu-Kapovich}.

\begin{Theorem-intro} \label{thm-intro-high-finiteness}
	For every $n \geq 1$, there are groups $\Gamma, \Lambda$ of type $F_n$ such that $\Gamma, \Lambda$ share a cocompact envelope, and: \begin{itemize}
		\item $\Gamma, \Lambda$ are solvable;
		\item $\Gamma$ is of finite rank and torsion-free;
		\item $\Lambda$ is neither of finite rank, nor virtually torsion-free.
	\end{itemize}
\end{Theorem-intro}

As a consequence of Gromov's polynomial growth theorem \cite{Gromov-pol-growth}, the class of finitely generated nilpotent groups is QI-rigid. The problem whether the class of polycyclic groups is QI-rigid has been open for some time \cite{Farb-Mosher-Prob, Shalom-Acta, Eskin-Fisher-Whyte-annou,Eskin-Fisher-icm}. Shalom showed that any infinite group QI to a polycyclic group has a positive first virtual Betti number \cite{Shalom-Acta}. Eskin--Fisher--Whyte conjectured that the class of polycyclic groups is QI-rigid \cite[Conjecture 1.2]{Eskin-Fisher-Whyte-annou,Eskin-Fisher-icm}. They showed that the class of cocompact lattices in the Lie group Sol is QI-rigid \cite{EFW-coarse-I, EFW-coarse-II}. Generalizations to certain higher dimensional metabelian Lie groups have been obtained by Peng \cite{Peng-I,Peng-II}. This gave rise to subsequent works by Dymarz--Xie \cite{Dymarz-Xie-GGD-16} and Grayevsky--Pallier \cite{Grayevsky-Pallier-Sol5}. Very recently, among a long list of announcements, a positive solution to the QI rigidity problem of polycyclic groups has been announced in \cite{OAI-QI-polycyclic}.

All this motivates the problem of  investigating QI-rigidity for classes of solvable groups that  contain the class of polycyclic groups, and are in a sense as close as possible to the class of polycyclic groups. Theorem \ref{thm-intro-not-QI-rigid} and Theorem \ref{thm-intro-high-finiteness} contribute to this problem, as they have the following consequences:

\begin{Corollary-intro} \label{cor-intro-QI-flexible}
The class of finitely generated solvable groups of finite rank  is not QI-rigid. More generally, for every $n \geq 1$, the class of solvable groups of finite rank of type $F_n$ is not QI-rigid.
\end{Corollary-intro}

As far as we know, this is the first instance of failure of QI-rigidity for finitely presented groups - and more generally groups with higher finiteness properties - in the realm of solvable groups. 

One can further ask about CE-rigidity for the class of solvable groups of finite rank of type $F_\infty$. Theorem \ref{thm-finite-rank-F-infty} below asserts that this class of groups is CE-rigid. So the phenomena exhibited in Theorem \ref{thm-intro-high-finiteness} for groups of type $F_n$ for $n \geq 1$ arbitrary cannot happen for groups of type $F_\infty$.

We point that that while in Theorem \ref{thm-intro-not-QI-rigid} the group $\Lambda$ is not virtually solvable, in Theorem \ref{thm-intro-high-finiteness} both $\Gamma$ and $\Lambda$ are solvable. It is an open question whether the class of finitely presented solvable groups is QI-rigid \cite[Question 4]{Farb-Mosher-Prob}. We also note that any two groups $\Gamma$ and $\Lambda$ as in Theorem \ref{thm-intro-high-finiteness} necessarily have the same Hirsch number, since a result of Sauer asserts that Hirsch number is a QI-invariant among solvable groups \cite{Sauer-GAFA-2006}.

We make some additional comments relating Theorem \ref{thm-intro-high-finiteness} with results from the literature: 

\begin{itemize}
	\item Farb--Mosher showed that the class of non-polycyclic finitely presented groups that are abelian-by-$\Z$ (such groups are necessarily of finite rank) is QI-rigid \cite{Farb-Mosher-abelian-by-cyclic}. The fact that the class finitely presented solvable groups of finite rank is not QI-rigid (i.e.\ Corollary  \ref{cor-intro-QI-flexible} for $n=2$) implies that this QI-rigidity result no longer holds beyond the abelian-by-$\Z$ case. 
	\item Shalom showed that if $\Gamma$ is solvable of finite rank, and if $\Lambda$ is a group QI to $\Gamma$ such that $\Lambda$  is solvable and torsion-free, then $\Lambda$ is of finite rank \cite[Theorem 1.6]{Shalom-Acta}. Theorem \ref{thm-intro-high-finiteness} shows that this theorem no longer holds without the assumption that $\Lambda$ is torsion-free. 
\end{itemize}

Examples from  \cite{Erschler-instab} show that the class of finitely generated torsion-free solvable groups is not QI-rigid. These examples are wreath products, and hence are not finitely presented \cite{Baumslag-wreath-fp}. Theorem \ref{thm-intro-high-finiteness} shows this failure of QI-rigidity persists even in the presence of finiteness properties:

\begin{Corollary-intro}
The class of finitely presented torsion-free solvable groups is not QI-rigid. More generally, for every $n \geq 2$ the class of torsion-free solvable groups of type $F_n$ is not QI-rigid.
\end{Corollary-intro}

\subsection*{Polycyclic groups}

We now turn our attention to polycyclic groups. The first statement of the following theorem asserts that when $\Gamma$ is polycyclic, the additional assumption on $\Lambda$ in Theorem  \ref{thm-intro-Rad-Lambda-fin} is not needed. We refer to \S \ref{subsec-stable-commable} for the definition of the relation of commability appearing in the second assertion.

\begin{Theorem-intro} \label{thm-intro-poly-commable}
The class of polycyclic groups is CE-rigid. More generally, every discrete group that is commable to a polycyclic group is virtually polycyclic. 
\end{Theorem-intro}

As mentioned above, very recently, and after the first version of this paper was posted on the arxiv, a major upgrading of Theorem  \ref{thm-intro-poly-commable} has been announced in \cite{OAI-QI-polycyclic}, stating that "commable" can be replaced by "QI".

Theorem  \ref{thm-intro-poly-commable} is a special case of Theorem  \ref{thm-commable-rigidity} below, which deals with a class of locally compact groups whose discrete members are precisely the virtually polycyclic groups. One consequence of Theorem  \ref{thm-commable-rigidity} is that if $\Gamma$ is polycyclic and $G$ is a cocompact envelope of $\Gamma$, then after modding out by a compact normal subgroup, $G$ is an extension of a connected Lie group and a discrete group. Hence another special case of Theorem  \ref{thm-commable-rigidity} (which is actually an intermediate step in the proof) is that every totally disconnected cocompact envelope of $\Gamma$ is compact-by-discrete. Recall that a group is compact-by-discrete if it has a compact normal subgroup whose associated quotient is discrete. Envelopes that are compact-by-discrete are somehow the trivial envelopes. 

For certain polycyclic groups $\Gamma$, we show that the only cocompact envelopes are the ones that live within a certain natural envelope. Let $d \geq 2$, and let $A \leq \mathrm{SL}(d,\Z)$ be a subgroup such that $A \simeq \Z^k$ and every non-trivial element of $A$ has $d$ distinct real positive eigenvalues. So $A$ is diagonalizable over $\R$, and there are commuting real matrices $X_1,\ldots,X_k $ such that, if $\varphi : \R^k \to \mathrm{GL}(d,\R)$ is the the  homomorphism defined by $\varphi(t_1,\ldots,t_k) = \exp(t_1 X_1 + \cdots + t_k X_k)$, then $\varphi(\Z^k) = A$.  The associated connected Lie group $G_\varphi = \R^d \rtimes \R^k$ is then a cocompact envelope of the group $\Gamma = \Z^d \rtimes A \simeq \Z^d \rtimes \Z^k$. We show the following:

\begin{Theorem-intro} \label{thm-intro-classification-env-polyc}
	Suppose that every non-trivial element of $A$ has $d$ distinct real eigenvalues, which are positive and different from $1$. Then every cocompact envelope $G$ of $\Gamma = \Z^d \rtimes A$ is, up to passing to a finite index subgroup and modding out by a compact normal subgroup, isomorphic to a closed cocompact subgroup of $G_{\varphi} =  \R^d \rtimes \R^k$. 
\end{Theorem-intro}

This theorem generalizes a result of Dymarz, who showed  Theorem  \ref{thm-intro-classification-env-polyc} for $k=1$. Specifically, when $d=2$ and $k=1$, the group $G_{\varphi} = \R^2 \rtimes \R$ is the Lie group Sol, and in that case Theorem \ref{thm-intro-classification-env-polyc} is \cite[Theorem 1.1.1]{Dymarz-env}. More generally, the case $d \geq 2$ and $k=1$ in Theorem  \ref{thm-intro-classification-env-polyc} is covered by \cite[Theorem 1.2]{Dymarz-env}. Dymarz's approach relies on fine geometric properties of model spaces, and makes use of rigidity results about their quasi-isometries from \cite{EFW-coarse-I,EFW-coarse-II,Peng-I,Peng-II}. Our approach is very different. We rely on structural restrictions on cocompact envelopes of polycyclic groups established in Theorem  \ref{thm-commable-rigidity}.

\subsection*{Convention} In the remainder of the paper we use the shorthand \textit{tdlc} for totally disconnected locally compact.

\subsection*{Outline and organisation}

A main tool in the proof of Theorem \ref{thm-intro-norma-nilp-same-tdlc-env}  is Theorem \ref{thm-abelien-norm-co} below. That result also plays a major role in the proof of Theorem  \ref{thm-intro-poly-commable}.  Theorem \ref{thm-abelien-norm-co} asserts that if $A$ is a finitely generated nilpotent subgroup of a tdlc group $G$, and $A$ has a cobounded normalizer in $G$ (see \S \ref{subsec-cobounded-cocompact} for the terminology), then $A$ always normalizes a compact open subgroup of $G$. That statement can be thought of as a poorness property of the dynamics of the global conjugation action of $A$ on $G$, as it means that there is an invariant compact neighbourhood of the identity. The proof of Theorem \ref{thm-abelien-norm-co} makes use of Willis' theory. The original form of Willis' work deals with the study of the conjugation action of individual elements on a tdlc group $G$. The important notion there is the notion of tidy subgroups. We give a very brief summary of their properties in \S \ref{subsec-tidy}. When we move in the study of the conjugation action on $G$ from the case of an individual acting element to the case of a subgroup $H$ acting on $G$, in general very few tools are available. However when $H$ is finitely generated nilpotent or polycyclic, results of Shalom--Willis ensure the existence of tidy subgroups common for all elements of $H$. The proof of Theorem \ref{thm-abelien-norm-co} notably relies on these results. It is given in Section \ref{sec-general-result-tdlc}. That section also contains the proof of Theorem \ref{thm-intro-norma-nilp-same-tdlc-env}. A further application of Theorem \ref{thm-abelien-norm-co}, also given in  Section \ref{sec-general-result-tdlc}, asserts that if a group $\Gamma$ has a finitely generated nilpotent normal subgroup with the additional assumption that $A$ is almost self-centralizing in $\Gamma$, then tdlc cocompact envelopes of $\Gamma$ are compact-dy-discrete (Theorem \ref{thm-gen-envelope-normal-nilp}).  

Section \ref{sec-polycyclic-commable} mostly deals with polycyclic groups. It contains the proof of Theorem \ref{thm-commable-rigidity} (of which Theorem  \ref{thm-intro-poly-commable} is a corollary), as well as the proof of Theorem \ref{thm-intro-classification-env-polyc}. This section also ends with another application of Theorem \ref{thm-abelien-norm-co}, concerning tdlc envelopes of lattices in Lie groups (\S \ref{sec-latt-Lie}).

Section \ref{sec-solv-finite rank} is dedicated to the proof of Theorem  \ref{thm-intro-Rad-Lambda-fin} about solvable groups of finite rank. Although Theorem  \ref{thm-intro-Rad-Lambda-fin} is a generalization of the first assertion of Theorem  \ref{thm-intro-poly-commable} about polycyclic groups, the proof takes a different approach. The first goal of Section \ref{sec-solv-finite rank} is to prove that if $G$ is a cocompact envelope of a finitely generated solvable group of finite rank, then the product of the connected component and the locally elliptic radical forms an open subgroup of $G$ (Theorem \ref{thm-gen-structure-envelopes}). The proof of that result crucially relies on works of Shalom and Cornulier--Tessera about property $H_{FD}$. Theorem  \ref{thm-intro-Rad-Lambda-fin} is deduced from Theorem \ref{thm-gen-structure-envelopes} together with the stability result Proposition \ref{prop-sequence-polycyclic} (which is the place where the absence of infinite locally finite normal subgroup is used). 

Our flexibility results that lead to Theorem \ref{thm-intro-not-QI-rigid} and Theorem \ref{thm-intro-high-finiteness} are proven in Section \ref{sec-flexibility}. Our two constructions consist in embedding some solvable groups of finite rank $\Gamma$ as cocompact lattices in some (necessarily amenable) locally compact group $G$, where $G$ has the structure of a product $G = G_c \times G_{td}$. The factor $G_c $ is a virtually connected solvable Lie group, and the factor $G_{td}$ is a totally disconnected group. The group $\Gamma$ is embedded as an irreducible lattice in $G$. On the other hand, things are arranged so that each one of the factors $G_c$ and $G_{td}$ admits a cocompact lattice such that the product of these produces a group $\Lambda$ that fails to be virtually solvable of finite rank. 

Our constructions involve the Diestel--Leader graphs $\DL_d(n)$, i.e.\ the subset of the product of $d$ copies of an $(n+1)$-regular tree $T_1 \times \cdots \times T_d$ defined by the equation $b_1 + \cdots + b_d = 0$, where $b_i$ is a Busemann function on $T_i$ (see \S \ref{subsec-prelim-DL-graphs}). 
 In our first construction, we have $G_c = \Isom(\R^2)$, the group of isometries of $\R^2$, and $G_{td} = \Isom(\DL_2(n))$. This construction provides the examples mentioned below Theorem \ref{thm-intro-not-QI-rigid}.  In our second construction, we have $G_c =\R^d \rtimes (\R^\times)^{d-1}$, where $(\R^\times)^{d-1}$ is identified with the group of $(d \times d)$-diagonal matrices of determinant one (so contrary to the previous one, $G_c $ has exponential growth). And  $G_{td} = \Isom(\DL_d(n_1)) \times \cdots \times \Isom(\DL_d(n_k))$. And $d \geq 2$ and $k \geq 1$ are arbitrary. Relying on work of Bartholdi--Neuhauser--Woess \cite{Bar-Neu-Woe}, this is this second construction that leads to Theorem \ref{thm-intro-high-finiteness}. As a concrete example, our smallest finitely presented groups $\Gamma, \Lambda$ as in Theorem \ref{thm-intro-high-finiteness} are obtained with $d=3$ and $k=1$ and are of the form $\Gamma = \Z[1/p]^3 \rtimes \Z^4$ and $\Lambda~=~\Z^3 \rtimes \Z^2 \times \F_p[t,t^{-1}, (t+1)^{-1} ] \rtimes \Z^{2}$.

\medskip

\noindent \textbf{Acknowledgments.} We are very grateful to Yves Cornulier and Gabriel Pallier for comments on a first version of this paper. 

\medskip

\noindent \textbf{Manuscript timeline and AI use.} This work was performed between 2022 and 2025. A first version was posted on the arxiv in October 2025 in order to make it available in view of the author's \textit{Habilitation à Diriger des Recherches} defense. Changes in the current version are the incorporation of comments from Yves Cornulier and Gabriel Pallier, addition of \cite{OAI-QI-polycyclic} to the bibliography, and a few other minor rewrites.

In April 2025 the author asked Chat GPT assistance to write the Maple code to provide explicit examples of matrices presented in Example \ref{ex-explicit-fin-pres} as well as to determine the norms of their $p$-adic  eigenvalues. Apart from that, no other use of AI was made.

\setcounter{tocdepth}{1}
\tableofcontents

\section{Preliminaries}

\textbf{General conventions.} We say that a group is (A)-by-(B) if it has a normal subgroup with (A) such that the quotient has (B). We denote by $Z(\Gamma)$ the center of a group $\Gamma$. If $A$ is a subgroup of $\Gamma$, then the centralizer of $A$ in $\Gamma$ is denoted $C_\Gamma(A)$.

\subsection{On cobounded and cocompact subgroups} \label{subsec-cobounded-cocompact}

We say that a non-necessarily closed subgroup $\Gamma$ of a locally compact group $G$ is {cobounded} in $G$ if there is a compact subset $K$ of $G$ such that $\Gamma K = G$. A subgroup is cocompact if it is closed and cobounded. Although we are mainly concerned with cocompact subgroups, it will often be more convenient to work with cobounded subgroups. One reason for that is the following fact that will be used repeatedly in the sequel: if $N$ is a closed normal subgroup of $G$ and $\Gamma$ a cobounded subgroup of $G$, then the image of $\Gamma$ in $G/N$ is a cobounded subgroup of $G/N$ (which fails in general for cocompact subgroups because the image is not necessarily closed). Another basic fact that will be used without further mention is that if $O$ is an open subgroup of $G$ and $\Gamma$ is cobounded (resp.\ cocompact) in $G$, then $\Gamma \cap O$ is cobounded (resp.\ cocompact) in $O$. 

The following proposition is already known. We include a proof for completeness.

\begin{Proposition} \label{prop-approx-fg}
	Let $\Gamma$ be a cobounded subgroup of a compactly generated locally compact group $G$. Then there exists a finitely generated subgroup $\Gamma'$ of $\Gamma$ such that $\Gamma'$ is cobounded in $G$.
\end{Proposition}

We will make use of the following lemma. 

\begin{Lemma} \label{lem-chab-neigh-cobounded}
Let $G$ be a group generated by a subset $U$ with $U = U^{-1}$. Suppose $S$ is a subset of $G$ such that $U^2 \subseteq SU$. Then $G = \langle S \rangle U$.
\end{Lemma}

\begin{proof}
The inclusion $U^2 \subseteq SU$ implies $U^n  \subseteq\langle S \rangle U$ for every $n \geq 1$. Since $G = \langle U \rangle$ and $U = U^{-1}$, we infer $G = \bigcup_{n \geq 1} U^n \subseteq  \langle S \rangle U$.
\end{proof}

\begin{proof}[Proof of Proposition \ref{prop-approx-fg}]
	The subgroup $H = \overline{\Gamma}$ is a closed cocompact subgroup of $G$, and hence $H$ is compactly generated since $G$ is compactly generated (see for instance \cite[Proposition 2.C.8(3)]{CorHar}). Since a subgroup of $H$ is cobounded in $H$ if and only if it is cobounded in $G$, without loss of generality one can assume $\Gamma$ is dense in $G$. Let $U$ be an open neighbourhood of $1$ in $G$ such that $U$ is relatively compact, $U = U^{-1}$ and $G = \langle U \rangle$. Since $\Gamma$ is dense, one has $G = \Gamma U$. In particular $\overline{U^2} \subseteq \Gamma U$. Since $\overline{U^2}$ is compact and $U$ is open, one can find a finite subset $S$ of $\Gamma$ such that $\overline{U^2} \subseteq S U$. We deduce from Lemma \ref{lem-chab-neigh-cobounded} that $G = \langle S \rangle U$, and hence $\Gamma' =  \langle S \rangle $ is cobounded in $G$.
\end{proof}

\subsection{Connected-by-compact groups}

The following theorem is due to Yamabe \cite[Theorem 4.6]{MZ55}. 

\begin{Theorem} \label{thm-Hilbert-fifth}
Let $G$ be a connected-by-compact locally compact group. Then every neighbourhood of $G$ contains a compact normal subgroup $K$ such that $G/K$ is a Lie group with finitely many connected components. 
\end{Theorem}

\subsection{On the locally elliptic radical}

A locally compact group is  locally elliptic if every compact subset is contained in a compact subgroup. For $G$ discrete, locally elliptic means locally finite, i.e.\  every finite subset is contained in a finite subgroup. Every locally compact group $G$ admits a unique largest locally elliptic closed normal subgroup, called the locally elliptic radical of $G$, and denoted $\RadLE(G)$ \cite{Platonov-66-LE}. The quotient $G/ \RadLE(G)$ has trivial locally elliptic radical.  For $G$ discrete, we write $\RadLF$ instead of $ \RadLE$.

\subsection{On compact normal subgroups}

We denote by $W(G)$ the subgroup of a locally compact group $G$ generated by compact normal subgroups of $G$. The properties \enquote{$W(G)$ is compact} and \enquote{$G$ has a maximal compact normal subgroup} are equivalent. Note also that when a maximal compact normal subgroup exists, it is necessarily unique. It may happen that $W(G)$ is closed but non-compact, and it may also happen that $W(G)$ is not closed. 

\begin{Proposition} \label{prop-Lie-WG-cpct}
If $G$ has a compact normal subgroup $K$ such that $G/K$ is  a Lie group with finitely many connected components, then $W(G)$ is compact. 
\end{Proposition}

\begin{proof}
See \cite[Theorem 3.1, XV]{Hochschild} for the Lie group case. The given property being invariant under extension with compact kernel, the statement follows.
\end{proof}

The topological FC-center of $G$, denoted $B(G)$, is the set of elements of $G$ with a relatively compact conjugacy class.  Clearly $W(G) \leq B(G)$. The following is a consequence of \cite[Theorem 4]{WuYu72}.

\begin{Proposition} \label{prop-struct-BG}
	If $G$ is tdlc group such that $W(G)$ is closed, then $W(G)$ is open in $B(G)$, and $B(G) / W(G)$ is a discrete torsion-free abelian group.
\end{Proposition}

\begin{Lemma} \label{lem-QZ-open-centr}
	Let $G$ be a tdlc group and $N$ a discrete normal subgroup of $G$. Then for every finitely generated subgroup $\Lambda$ of $N$, $C_G(\Lambda)$ is an open subgroup of $G$.
\end{Lemma}

\begin{proof}
Each element  $g \in N$ has a discrete conjugacy class, and hence an open centralizer. Hence if $\Lambda = \langle g_1, \ldots, g_n \rangle$ then $C_G(\Lambda) = \bigcap_i C_G(g_i)$ is open. 
\end{proof}

\begin{Proposition} \label{prop-fg-subgroup-B(G)}
Let $G$ be a $\sigma$-compact tdlc group such that $W(G)$ is closed. Let $\Lambda$ be a finitely generated subgroup of $B(G)$. Then $\Lambda$ normalizes a compact open subgroup of $G$.
\end{Proposition}

\begin{proof}
As proven in the proof of Theorem 2.8 in \cite{Cor-comm-focal}, since $G$ is $\sigma$-compact and $W(G)$ is closed, there exists $K$ compact normal in $G$ such that $K$ is open in $W(G)$. Since $\Lambda$ normalizes a compact open subgroup of $G$ if and only if the image of $\Lambda$ in $G/K$ normalizes a compact open subgroup of $G/K$, it is therefore enough to prove the statement when $W(G)$ is discrete. In that case $B(G)$ is also discrete by Proposition \ref{prop-struct-BG}. Since $\Lambda$ is finitely generated, the statement follows from Lemma \ref{lem-QZ-open-centr}. 
\end{proof}

The following follows from \cite[Theorem 5.5]{Wang-compactness}. 

\begin{Proposition} \label{prop-polycompact-cocompact}
	Let $G$ be a locally compact group, and $H$ a  cocompact subgroup of $G$. Then every compact normal subgroup of $H$ is contained in a compact normal subgroup of $G$. 
\end{Proposition}

\subsection{On commensurators} \label{subsec-terminology-commens}

A virtual isomorphism of a group $A$ is an isomorphism between finite index subgroups of $A$. Two virtual isomorphisms are $\sim$ if they coincide on some finite index subgroup of $A$. Virtual isomorphisms modulo $\sim$ form a group, called the abstract commensurator of $A$, and denoted $\Comm(A)$. 

Let now $\Gamma$ be a group, and $A$  a subgroup of $\Gamma$. The (relative) commensurator of $A$ in $\Gamma$ is the set of $\gamma \in \Gamma$ such that $A$ and $\gamma A \gamma^{-1}$ are commensurable. It is a subgroup of $\Gamma$, denoted $\Comm_\Gamma(A)$. The subgroup $A$ is called commensurated in $\Gamma$ if $\Comm_\Gamma(A) = \Gamma$. For $\gamma \in \Comm_\Gamma(A)$, conjugation by $\gamma$ induces a virtual isomorphism $A \cap \gamma^{-1} A \gamma \to \gamma A \gamma^{-1} \cap A$ of $\Gamma$. This defines a group homomorphism $c_{\Gamma,A}: \Comm_\Gamma(A) \to \Comm(A)$. 

If now $G$ is a locally compact group, $\Comm(G)$ is defined similarly, except that we consider topological isomorphisms between finite index open subgroups of $G$, modulo coincidence on some finite index open subgroup.

\section{Nilpotent subgroups with cobounded normalizers in tdlc groups} \label{sec-general-result-tdlc}

The setting of this section is that of an ambient tdlc group $G$ having a cocompact (or more generally cobounded) subgroup $\Gamma$ such that $\Gamma$ admits a normal finitely generated nilpotent subgroup $A$. The main results of this section are Theorem \ref{thm-abelien-norm-co}, Theorem \ref{thm-norma-nilp-same-tdlc-env} (which is Theorem \ref{thm-intro-norma-nilp-same-tdlc-env} from the introduction) and Theorem \ref{thm-gen-envelope-normal-nilp}. 

Before moving further, we mention that Bader--Furman--Sauer showed that, under certain group theoretic requirements on a group $\Gamma$, the envelopes $G$ of $\Gamma$ in which the connected component of the identity $G^0$ is not compact, enjoy very strong restrictions. See \cite[Theorem A]{BFS-envelopes}. This result reduces, for those groups $\Gamma$, the general study of envelopes of $\Gamma$ to the study of tdlc envelopes. We also note that among those aforementioned group theoretic requirements on $\Gamma$ for \cite[Theorem A]{BFS-envelopes} to hold, there is the fact that $\Gamma$ does not admit any infinite amenable commensurated subgroup. Hence the setting of the present section, where the ambient group $G$ is tdlc and the subgroup $\Gamma$ does admit an infinite normal amenable (indeed nilpotent) subgroup, is disjoint from (and hence complementary to) the setting of \cite[Theorem A]{BFS-envelopes}. 

We also mention that the class of groups $\Gamma$ containing a commensurated nilpotent subgroup naturally appears in \cite{Margolis-commensurated-abelian}, where Margolis showed that given a finitely generated group $\Gamma$ with a commensurated (e.g.\ normal) finitely generated nilpotent subgroup $A$, one can always construct a cocompact envelope $G$ of $\Gamma$ such that the connected component of the identity $G^0$ is the real Malcev completion of (a torsion-free finite index subgroup of) the group $A$ \cite[Theorem 5.2, Remark 3.8]{Margolis-commensurated-abelian}. 

\subsection{Preliminaries about tidy subgroups} \label{subsec-tidy}

We will make use of Willis' theory of tidy subgroups in tdlc groups, through Theorem \ref{thm-nilp-polycyclic-common-tidy} and Theorem \ref{thm-common-tidy-derived-normalizes}. For the reader's convenience we give a brief account of the objects under consideration. 

An automorphism $\alpha$ of a locally compact group $G$ is called a compaction if there exists a neighbourhood of the identity $\Omega$ in $G$ such that for every compact subset $K$ of $G$, there is $n_0 \geq 1$ such that $\alpha^n(K) \subset \Omega$ for every $n \geq n_0$. When this holds, there is a unique maximal compact subgroup $L$ of $G$ such that $\alpha(L) = L$, called the limit group of the compaction. The subgroup $L$ consists of the set of elements of $G$ whose two-sided $\alpha$-orbit is relatively compact. See \cite[Lemma 6.3]{CCMT}. 

If $\alpha$ is an automorphism of a tdlc  group $G$, the contraction group of $\alpha$ is \[\mathrm{con}(\alpha) = \left\lbrace x \in G \, : \, \alpha^n(x) \rightarrow 1 \, \text{when} \, n \to \infty \right\rbrace. \] The subgroup $\mathrm{con}(\alpha) $ need not be closed, and we are led to consider $C_\alpha = \overline{\mathrm{con}(\alpha) }$, which is a closed $\alpha$-invariant subgroup of $G$. So $\alpha$ induces an automorphism of the tdlc group $C_\alpha$, and $\alpha$ acts on $C_\alpha$ as a compaction \cite[Proposition 6.17]{CCMT}. The associated limit group is denoted $K_\alpha$ (called the nub of $\alpha$ in Willis' papers). 

If $U$ is a compact open subgroup of $G$, let $U_- = \bigcap_{n \geq 1} \alpha^{-n}(U)$, $U_+ = \bigcap_{n \geq 1} \alpha^{n}(U)$. The subgroups $U_-$ and  $U_+$  are compact, $U_-$ is the largest subgroup of $U$ such that $\alpha(U_-) \leq U_-$; $U_+$ is the largest subgroup of $U$ such that $U_+ \leq \alpha(U_+)$. The subgroup $U_0 = U_- \cap U_+$ is then the largest subgroup of $U$ such that $\alpha(U_0) = U_0$. We define $U_{--} = \bigcup_{n \geq 1} \alpha^{-n}(U_-)$ and $U_{++} = \bigcup_{n \geq 1} \alpha^{n}(U_+)$. Clearly $\mathrm{con}(\alpha) \leq U_{--}$. 

\begin{Definition}
A compact open subgroup $U$ of $G$ is called {tidy} for $\alpha$ if it satisfies: \begin{enumerate}
	\item \label{item-U--closed} $U_{--}$ is closed;
	\item \label{item-U-U+} $U = U_- U_+$.
\end{enumerate}
\end{Definition}

Tidy subgroups always exist \cite[Theorem 1]{Willis-1994-MathAn}. When (\ref{item-U--closed}) holds, the automorphism $\alpha$ acts on $U_{--}$ as a compaction, and $U_{--}$ contains $C_\alpha$ as a  cocompact subgroup \cite[Corollary 3.17]{Baumg-Willis}. So the compact open subgroups $U$ that verify (\ref{item-U--closed}) are the ones for which the associated subgroup $U_{--}$ \enquote{encapsulates well} the compacting part $C_\alpha$.  The limit group of the compaction $U_{--}$ is equal to $U_0$. This means that the elements of $U_{--}$ that have a two-sided $\alpha$-orbit that is relatively compact must lie in the compact $\alpha$-invariant subgroup $U_0$. Under the additional condition (\ref{item-U-U+}), more is true: an element $x \in U$ has a two-sided $\alpha$-orbit that is relatively compact only if $x$ lies in $U_0$ \cite[Lemma 9]{Willis-1994-MathAn}.

The scale of $\alpha$ is the minimum value of the index $(\alpha(U) : U \cap \alpha(U))$ when $U$ ranges over compact open subgroups of $G$. The compact open subgroups that realize this minimum value are exactly the tidy subgroups of $\alpha$ \cite{Willis-2001-JA}. There is a compact open subgroup $U$ such that $\alpha(U) = U$ if and only if $\alpha$ and $\alpha^{-1}$ have scale one. When this holds, the tidy subgroups are the $U$ such that $\alpha(U) = U$. 

In the sequel we use these notions for elements of $G$, viewed as inner automorphisms. The following theorem is due to Shalom--Willis  \cite[Theorems 4.9--4.13]{Shalom-Willis}. 

\begin{Theorem} \label{thm-nilp-polycyclic-common-tidy}
	Let $G$ be a tdlc group. The following hold: \begin{enumerate}
		\item If $A$ is a finitely generated nilpotent subgroup of $G$, then there exists a compact open subgroup $U$ of $G$ such that $U$ is tidy for every element of $A$. 
		\item If $P$ a polycyclic subgroup of $G$, then there exist  a finite index subgroup $S$ of $P$ and a compact open subgroup $U$ of $G$  such that $U$ is tidy for every element of $S$. 
	\end{enumerate} 
\end{Theorem}

We will notably use the first item of Theorem \ref{thm-nilp-polycyclic-common-tidy} through the following corollary:

\begin{Corollary} \label{cor-fg-nilpo-normalizes-co}
	Let $G$ be a tdlc group, and $A = \langle a_1, \ldots, a_n \rangle$ be a nilpotent subgroup of $G$ such each $a_i$ normalizes a compact open subgroup $U_i$ of $G$. Then there exists a compact open subgroup $U$ of $G$ that is normalized by $A$. 
\end{Corollary}

\begin{proof}
Since $A$ is finitely generated nilpotent, according to Theorem \ref{thm-nilp-polycyclic-common-tidy} one can find a compact open subgroup $U$ that is tidy for every element of $A$. In particular $U$ is tidy for $a_1, \ldots, a_n$. Since  $a_i$ normalizes a compact open subgroup, that $U$ is tidy for $a_i$ means $a_i$ normalizes $U$. Therefore so does $A = \langle a_1, \ldots, a_n \rangle$.
\end{proof}

The following theorem is due to Willis \cite[Theorem 4.15]{Willis-NYJM}. 

\begin{Theorem} \label{thm-common-tidy-derived-normalizes}
Let $G$ be a tdlc group, and let $H$ be a subgroup of $G$ such that  there exists a compact open subgroup $U$ of $G$ such that $U$ is tidy for every element of $H$. Then $[H,H]$ normalizes $U$. 
\end{Theorem}

\subsection{The proof of of Theorem \ref{thm-abelien-norm-co}}

We recall the following result from the literature. 

\begin{Theorem} \label{thm-tdlc-poly-growth}
	Suppose that $G$ is a compactly generated tdlc group containing a closed cocompact subgroup $H$ such that $H$ is nilpotent. Then $W(G)$ is compact and $G/W(G)$ is discrete virtually nilpotent. 
\end{Theorem}

\begin{proof}
Since $H$ is nilpotent, $H$ has polynomial growth \cite[Theorem II.4]{Guivarch-croissance}. Since $H$ is cocompact in $G$, $G$ also has polynomial growth \cite[Theorem I.4]{Guivarch-croissance}. So by the locally compact generalization of Gromov's theorem \cite{Losert-growth-I}, $G$ is compact-by-(discrete virtually nilpotent). The discrete quotient has a maximal finite normal subgroup, and its pre-image in $G$ is $W(G)$.
\end{proof}

Recall that any two compact open subgroups of a tdlc group $G$ are commensurable. In particular every  compact open subgroup  of $G$ is commensurated in $G$.

\begin{Lemma} \label{lem-UA-commens}
Let $G$ be a tdlc group, and $\Gamma$ a cobounded subgroup of $G$. Let $A$ be a subgroup of $G$ that is commensurated by $\Gamma$. Suppose that $A$ normalizes a compact open subgroup $U$ of $G$, and let $O := UA$. Then $\Gamma \leq \Comm_G(O)$ and $\Comm_G(O)$ is a finite index subgroup of $G$.
\end{Lemma}

\begin{proof}
	For $\gamma \in \Gamma$, we have that $\gamma^{-1} U  \gamma \cap U$ is normalized by $\gamma^{-1} A  \gamma \cap A$. Hence $(\gamma^{-1} U  \gamma \cap U) (\gamma^{-1} A  \gamma \cap A)$ is a subgroup of $O$, which has finite index in $O$ because $\gamma^{-1} U  \gamma \cap U$ and $\gamma^{-1} A  \gamma \cap A$ have finite index in $U$ and $A$ because $U$ and $A$ are commensurated by $\Gamma$. So $\Gamma \leq \Comm_G(O)$. Since $O \leq \Comm_G(O)$, the subgroup $\Comm_G(O)$ is also open. So $\Comm_G(O)$ is a finite index subgroup of $G$.
\end{proof}

\begin{Proposition} \label{prop-commens-upgrade-normal}
Let $G$ be a tdlc group, and $\Gamma$ a cobounded subgroup of $G$. Let $A$ be a subgroup of $G$ that is normalized by $\Gamma$, and suppose that for each positive integer $n$, $A$ has only finitely many subgroups of index $n$. If $A$ normalizes a compact open subgroup of $G$, then there exists an open finite index subgroup $G'$ of $G$ containing $A$ and $\Gamma$, and a closed normal subgroup $M$ of $G'$ such that $A \leq M$ and $A$ is cobounded in $M$.
\end{Proposition}

\begin{proof}
Let $U$ be a compact open subgroup of $G$ normalized by $A$. Let $O := UA$ and $G' := \Comm_G(O)$.  By Lemma  \ref{lem-UA-commens} we have $\Gamma \leq G'$ and $G'$ is a finite index subgroup of $G$. Let $g_1, \ldots g_n \in G'$ such that $G' = \bigcup U g_i \Gamma$. Let \[O' := \bigcap_{i=1}^n g_i^{-1} O g_i \, \, \text{and} \, \, A' := O' \cap A.\] Since $g_i \in G'$ for every $i$, $O'$ is a finite index open subgroup of $O$, and hence $A'$ is a finite index subgroup of $A$. Moreover one has $g_i A' g_i^{-1} \subseteq O$ for all $i$. By the assumption finite index subgroups of $A$, one can find a characteristic subgroup $B$ of  $A$ such that $B \leq A'$. The subgroup $B$ keeps the property that $g_i B g_i^{-1} \subseteq O$ for all $i$, and $B$ is in addition $\Gamma$-invariant. The decomposition $G' = \bigcup U g_i \Gamma$ then shows that $gBg^{-1} \subseteq O$ for all $g$ in $G'$. The subgroup $N := \cap_{g \in G'} g O g^{-1}$ is therefore a closed normal subgroup of $G'$ such that $B \leq N$. Since $B$ has finite index in $A$, the image of $A$ in $G'/N$ is a discrete finite subgroup $F$. Moreover $F$ has a cocompact normalizer in $G'/N$. Proposition \ref{prop-polycompact-cocompact} implies $F$ is contained in a compact normal subgroup $K$ of $G'/N$. The pre-image $M$ of $K$ in $G'$ is a normal subgroup of $G'$ in which $A$ is contained and cobounded. 
\end{proof}

The following will be a key tool later in the paper. 

\begin{Theorem} \label{thm-abelien-norm-co}
Let $G$ be a tdlc group, and $\Gamma$ a cobounded subgroup of $G$. If $A$ is a finitely generated nilpotent normal subgroup of $\Gamma$, then $A$ normalizes a compact open subgroup of $G$.
\end{Theorem}

\begin{proof}
Suppose for a moment that we have proved the statement under the assumption that $G$ is compactly generated. Let $G$ be an arbitrary tdlc group, and $\Gamma,A$ as in the statement. Let $O$ be a compactly generated open subgroup of $G$ containing $A$. Then $\Gamma \cap O$ is a cobounded subgroup of $O$, which surely normalizes $A$. Since $O$ is compactly generated, by the assumption $A$ normalizes a compact open subgroup of $O$. Since $O$ is open in $G$, the statement follows.

Hence it is enough to prove the statement assuming that  $G$ is compactly generated. (We can also assume that $A$ torsion-free). We argue by contradiction. If the statement is not true, choose a counter-example $(G,\Gamma,A)$ with $A$ of minimal Hirsch number. 
	
	We first claim that no non-trivial element of $A$ normalizes a compact open subgroup of $G$. Indeed, suppose that $a_0$ is a non-trivial element of $A$ that normalizes a compact open subgroup of $G$. Let $B$ be the subgroup of $A$ generated by all the $\Gamma$-conjugates of $a_0$. Then $B$ is a finitely generated nilpotent group, and $B$ is normalized by $\Gamma$. The subgroup $B$ is finitely generated nilpotent, and generated by elements that normalize a compact open subgroup. By Corollary \ref{cor-fg-nilpo-normalizes-co} this implies that the entire subgroup $B$ normalizes a compact open subgroup of $G$. By Proposition  \ref{prop-commens-upgrade-normal}, there is a finite index subgroup $G'$ of $G$ and a closed normal subgroup $M$ of $G'$ such that $B$ is contained in $M$ and $B$ is cobounded in $M$. The quotient $Q = G' /M$ remains compactly generated. The subgroup $A/A \cap M$ has  Hirsch number strictly smaller than the one of $A$, so by minimality $A/A \cap M$ normalizes a compact open subgroup $V$ of $Q$. Let $O$ be the preimage of $V$ in $G$. The subgroup $O$ is normalized by $A$. The subgroup $B$ is cobounded in $O$ because $B$ is cobounded in $M$ and $M$ is cocompact in $O$. So $O$ is a tdlc group with a closed cocompact compactly generated nilpotent subgroup. By Theorem \ref{thm-tdlc-poly-growth} this implies that $O$ is compact-by-discrete with a unique maximal compact open normal subgroup $U$. In particular $U$ is characteristic in $O$. Since $A$ normalizes $O$, we deduce that $A$ also normalizes $U$. Contradiction.
	
	Now we claim that the image of $\Gamma$ in $\Aut(A)$ cannot be a torsion group. Indeed, suppose for a contradiction that this is the case. Since $A$ is finitely generated and nilpotent, the group $\Aut(A)$ is linear over $\mathbb{Z}$, i.e. is a subgroup of $\mathrm{GL}(n,\mathbb{Z})$ for some $n$. Since every torsion subgroup of $\mathrm{GL}(n,\mathbb{Z})$ is finite, it follows that $\Gamma$ virtually centralizes $A$. Since $\Gamma$ is cobounded in $G$, this implies that $A$ is contained in $B(G)$. According to Proposition \ref{prop-fg-subgroup-B(G)}, $A$ normalizes a compact open subgroup, which is not true since we assume $A$ is a counterexample. So we have reached a contradiction. Here we can indeed apply Proposition \ref{prop-fg-subgroup-B(G)}, because the assumption that $G$ is compactly generated implies that $G$ is  $\sigma$-compact and that $W(G)$ is closed by (see the references to Trofimov and Möller in \cite{Cor-comm-focal}). 
	
	Hence there is an element $\gamma$ in $\Gamma$ such that $\gamma$ acts on $A$ as an element of infinite order. Let $P$ be the subgroup generated by $A$ and $\gamma$. The subgroup $P$ is polycyclic, and by definition of $\gamma$ we have that $P$ is not virtually abelian. According to Theorem  \ref{thm-nilp-polycyclic-common-tidy}, $P$ admits a finite index subgroup $S$ such that there exists a compact open subgroup $U$ of $G$ that is tidy for every element of $S$. Theorem \ref{thm-common-tidy-derived-normalizes} applied to $S$ and $U$ therefore implies that $[S,S]$ normalizes $U$. Since $[S,S] \leq A$ and $[S,S]$ is non-trivial as $P$ is not virtually abelian, we have reached a contradiction with the beginning of the proof.
\end{proof}

\begin{Remark}
	When $A$ is not  finitely generated, Theorem \ref{thm-abelien-norm-co} is no longer true. For instance, the wreath product $\Gamma = \Z/n \Z \wr \Z$ embeds as a discrete and cocompact subgroup in the isometry group $G$ of the the Diestel--Leader graph $\DL_2(n)$ (see Section \ref{sec-flexibility}). The subgroup $A = \bigoplus \Z/n \Z$ is an abelian normal subgroup of $\Gamma$, and $A$ does not normalize any compact open subgroup of $G$.
\end{Remark}

\subsection{The proof of Theorem \ref{thm-intro-norma-nilp-same-tdlc-env}} \label{subsec-proof-Theorem1}

Let $A$ be a  finitely generated torsion-free nilpotent group. The only automorphism of $A$ which coincides with the identity on a finite index subgroup of $A$ is the trivial automorphism. Hence the homomorphism $\Aut(A) \to \Comm(A)$ is injective. In the sequel we often identify $\Aut(A)$ with its image in $\Comm(A)$. 

\begin{Notation}
 For $k \geq 1$, we denote by $kA$ the subgroup of $A$ generated by $k$-powers.
\end{Notation}

The subgroup $kA$ is a characteristic subgroup of $A$, which is of finite index in $A$ \cite[2.3.2]{Lennox-Rob}.

\begin{Lemma} \label{lem-caract-aut-nilpo}
	Let $A$ be a finitely generated torsion-free nilpotent group. Let $A'$ a finite index subgroup of $A$, and let $k \geq 1$. If $f: k A \to A'$ is an isomorphism such that the image of $f$ in $\Comm(A)$ belongs to $\Aut(A)$, then $f$ extends to an automorphism of $A$, and $A' = k A$.
\end{Lemma}

\begin{proof}
By the assumption there exist  $\alpha \in \Aut(A)$ and a finite index subgroup $A''$ of $A$ such that $A'' \leq kA$ and $f$ and $\alpha$ coincide on $A''$. Take $a \in k A$, and take $n \geq 1$ such that $a^n \in A''$. Then $f(a)^n = f(a^n) = \alpha(a^n) = \alpha(a)^n$. Now in a torsion-free nilpotent group, two elements with the same $n$-power are necessarily equal \cite[2.1.2]{Lennox-Rob}. So $f(a) = \alpha(a)$, and $\alpha$  is an automorphism of $A$ that extends $f$. In particular $A' = \alpha(k A) = k A$.
\end{proof}

Recall from \S  \ref{subsec-terminology-commens} that if $A$ is a commensurated subgroup of a group $\Gamma$, we denote by $c_{\Gamma,A}: \Gamma \to \Comm(A)$ the associated homomorphism. 

\begin{Proposition} \label{prop-comm-nilp-image-in-Aut}
	Let $\Gamma$ be a finitely generated group with a commensurated subgroup $A$ such that $A$ is finitely generated, torsion-free and nilpotent. Suppose that the image of $c_{\Gamma,A}: \Gamma \to \Comm(A)$ lies in $\Aut(A)$. Then there is a finite index subgroup of $A$ that is normal in $\Gamma$. 
\end{Proposition}

\begin{proof}
Let $S = \left\lbrace \gamma_1, \ldots, \gamma_n \right\rbrace 	$ be a generating subset of $\Gamma$. The subgroups $\left\lbrace A \cap \gamma_i^{-1} A \gamma_i \right\rbrace $ form a finite collection of finite index subgroups of $A$. Hence one can find $k \geq 1$ such that $k A \leq A \cap \gamma_i^{-1} A \gamma_i$ for every $i$. Applying Lemma \ref{lem-caract-aut-nilpo} to the restriction to $k A$ of the conjugation by $\gamma_i$, we infer that $k A$ is normalized by $\gamma_i$ for every $i$. Hence  $k A$ is normal in $\Gamma$.
\end{proof}

\begin{Lemma} \label{lem-comm-mod-W}
Let $O$ be a locally compact group such that $W(O)$ is closed. Then the homomorphism $O \to O / W(O)$ induces a  homomorphism $\psi:  \Comm(O) \to \Comm(O/W(O))$.  
\end{Lemma}

\begin{proof}
Since $W(O)$ is closed, $O/W(O)$ is a locally compact group. Let $\mathcal{K}$ denote the collection of compact normal subgroups of $O$, so that $W(O)$ is the union of all members of $\mathcal{K}$. If $O_1$ is a finite index open subgroup of $O$, then we have \[ W(O) \cap O_1 = \bigcup_\mathcal{K} K  \cap O_1 = W(O_1), \] where the last equality follows from Proposition  \ref{prop-polycompact-cocompact}.

If $f : O_1 \to O_2$ is an isomorphism between two open finite index subgroups of $O$, then $f$ sends $W(O_1) = W(O) \cap O_1$ to $W(O_2)  = W(O) \cap O_2$, and hence induces an isomorphism $\overline{f} : O_1 / W(O) \cap O_1 \to O_2/ W(O) \cap O_2$. Since $O_i$ is open in $O$, $O_i / W(O) \cap O_i$ is isomorphic to a finite index open subgroup of $O/W(O)$. If $f$ is the identity on a finite index open subgroup of $O$ then $\overline{f}$ is the identity on a finite index open subgroup of $O/W(O)$, so that $f \mapsto \overline{f}$ defines a map $\psi:  \Comm(O) \to \Comm(O/W(O))$. Moreover composition behaves as expected, so that $\psi$ is a homomorphism. 
\end{proof}

\begin{Proposition} \label{prop-extend-comm-homomorphism}
	Let $\Gamma$ be a group with a commensurated subgroup $A$ such that $A$ has no non-trivial finite normal subgroup. Let $G$ be a tdlc cocompact envelope of $\Gamma$ such that $A$ normalizes a compact open subgroup $U$ of $G$, and let $O = UA$. Then the following hold: \begin{enumerate}
		\item $G' := \Comm_G(O)$ is a finite index subgroup of $G$, and $\Gamma \leq G'$;
		\item $W(O) = U$ and $O/U \simeq A$;
		\item If $\psi: \Comm(O) \to \Comm(O/U)$ is the homomorphism from Lemma \ref{lem-comm-mod-W}, and if $\varphi = \psi \circ c_{G',O}: G' \to \Comm(O/U)$, then the restriction of $\varphi $ to $\Gamma$ is equal to $c_{\Gamma,A}: \Gamma \to \Comm(A)$ (after identifying $O/U$ with $A$);
		\item $\varphi $ has open kernel, and the image of $\varphi$ contains the image of $c_{\Gamma,A}$ as a finite index subgroup. 
	\end{enumerate}
\end{Proposition}

\begin{proof}
Lemma  \ref{lem-UA-commens} asserts $G'=\Comm_G(O)$ is a finite index subgroup of $G$, and $\Gamma \leq G'$. Without loss of generality in the remaining of the proof we can assume $\Comm_G(O) = G$.

	Since $A$ is discrete in $G$ and $U$ is compact, $U \cap A$ is finite. Moreover $U \cap A$ is normal in $A$. So by the assumption on $A$, we have $U \cap A = 1$, $O \simeq U \rtimes A$, and $U = W(O)$. Let $\psi: \Comm(O) \to \Comm(O/U)$ be the natural homomorphism (Lemma \ref{lem-comm-mod-W}), and let $\varphi : G \to \Comm(O/U)$ be the composition of $c_{G,O}: G \to \Comm(O)$ and $\psi : \Comm(O) \to \Comm(O/U)$. We check that $\varphi $ verifies the required properties. For $\gamma \in \Gamma$, as observed in the proof of Lemma  \ref{lem-UA-commens}, conjugation by $\gamma$ induces an isomorphism between the open subgroups $O_1 = (\gamma^{-1} U  \gamma \cap U) (\gamma^{-1} A  \gamma \cap A)$ and $O_2 = (U \cap \gamma U  \gamma ^{-1}) (A \cap \gamma  \gamma^{-1} A)$ of $O$. Hence by construction we have $\varphi(\gamma) = \psi(c_{G,O}(\gamma)) =  c_{\Gamma,A}(\gamma)$. Hence $\varphi$ indeed extends $c_{\Gamma,A}$. Moreover, since $U$ is normal in $O$, conjugation by $U$ induces the trivial map at the level of $O/U$. So $\varphi(U) = \psi(c_{G,O}(U)) = 1$, and the kernel of $\varphi $ is therefore open. In particular the subgroup $\ker(\varphi) \Gamma$ has finite index in $G$, which means that  the image of $\varphi$ contains the image of $c_{\Gamma,A}$ as a finite index subgroup.
\end{proof}

The following is Theorem \ref{thm-intro-norma-nilp-same-tdlc-env} from the introduction.

\begin{Theorem} \label{thm-norma-nilp-same-tdlc-env}
	Let	$\Gamma$ be a finitely generated group with a normal subgroup $A \lhd \Gamma$ such that $A$ is finitely generated and nilpotent. Suppose that $\Gamma$ and $\Lambda$ share a tdlc cocompact envelope. Then there is a finite index subgroup $\Lambda'$ of $\Lambda$  such that $\Lambda'$ admits a normal subgroup $B \lhd \Lambda'$ such that $A$ and $B$ are virtually isomorphic. 
\end{Theorem}

\begin{proof}
Upon changing $A$ by a finite index subgroup of $A$, we can assume $A$ is torsion-free. Let $G$ be a tdlc group that contains both $\Gamma$ and $\Lambda$ as discrete cocompact subgroups. 	By Theorem \ref{thm-abelien-norm-co} applied to $(G,\Gamma)$, there exists a compact open subgroup $U$ of $G$ such $A$ normalizes $U$. Let $O = UA \simeq U \rtimes A$.  By Lemma  \ref{lem-UA-commens}, $\Comm_G(O)$ is a finite index subgroup of $G$. Hence upon replacing $G$ by $\Comm_G(O)$ and $\Lambda$ by $\Lambda \cap \Comm_G(O)$, we can assume $O$ is commensurated in $G$.

The subgroup $\Lambda \cap O$ is discrete and cocompact in $O$, and $\Lambda \cap O$ is finite-by-nilpotent. Hence $\Lambda \cap O$ is virtually nilpotent, and we can find a  torsion-free subgroup $B \leq \Lambda \cap O$ such that $B$ has finite index in $\Lambda \cap O$. In particular $B$ intersects $U$ trivially since $B$ is discrete in $G$, and hence $B$ is  isomorphic to a finite index subgroup of $A$.

Since $O$ is commensurated in $G$, $\Lambda \cap O$ (and hence $B$) is commensurated in $\Lambda$.  We apply Proposition \ref{prop-extend-comm-homomorphism} on the one hand to $(G,\Gamma,A)$, and on the other hand to $(G,\Lambda,B)$. Let $O' := U B \simeq U \rtimes B$, which is a finite index open subgroup of $O$.  We have the following diagram \[ \begin{tikzcd}
G \arrow [r, "c_{G,O}"] & \Comm(O) \arrow[r]  & \Comm(O/U) \arrow[r]  & \Comm(A) \\
G \arrow[r, "c_{G,O'}"] & \Comm(O') \arrow[u] \arrow[r] &  \Comm(O'/U) \arrow[u] \arrow[r] &  \Comm(B)  \arrow[u, "\eta"]
\end{tikzcd} \]
The middle arrows in horizontal lines are the homomorphisms given by Lemma \ref{lem-comm-mod-W}, and the right arrows are the isomorphisms induced by reduction modulo $U$. The left and middle vertical arrows are the isomorphisms induced by the fact that $O'$ (resp.\ $O'/U$) has finite index in $O$ (resp.\ $O/U$). The right vertical arrow makes the right square (entirely made of isomorphisms) commute.

Proposition \ref{prop-extend-comm-homomorphism} says that the restriction to $\Gamma$ of the first line is equal to $c_{\Gamma,A}: \Gamma \to \Comm(A)$, and the image of $G$ in $\Comm(A)$ contains the image of $c_{\Gamma,A}$ as a finite index subgroup. And similarly for $\Lambda$ in the second line. Now since $A$ is normal in $\Gamma$, the image of $c_{\Gamma,A}$ lies in $\Aut(A)$. Hence $G$ virtually maps to $\Aut(A)$ in the first line. Since reduction modulo $U$ identifies $B$ with a finite index subgroup of $A$, and since a finite index subgroup of $A$ has a finite $\Aut(A)$-orbit, we infer that $\eta^{-1}(\Aut(B))$ contains $\Aut(B)$ as a finite index subgroup. Hence $G$ virtually maps to $\Aut(B)$ in the second line. By restricting to $\Lambda$ we infer that the image of $c_{\Lambda,B}$ is virtually contained in $\Aut(B)$. By Proposition \ref{prop-comm-nilp-image-in-Aut} applied to $\Lambda$, this implies that there is a finite index subgroup $\Lambda'$ of $\Lambda$ such that $\Lambda'$ normalizes a finite index subgroup $B'$ of $B$. 
\end{proof}

\begin{Remark}
Let $\Lambda = \Z^2 \rtimes \Z/4\Z$, where the  $\Z/4\Z$-action is via the rotation of angle $\pi /2$. Let also $\Gamma = \Z^2$ and $A = \Z \times \left\lbrace 0 \right\rbrace \lhd \Gamma$. Then here $\Gamma$ is even a finite index subgroup of $\Lambda$, and $\Lambda$ does not have any normal subgroup virtually isomorphic to $\Z$. Hence already for virtually isomorphic groups $\Gamma, \Lambda$ it is necessary to pass to a finite index subgroup of $\Lambda$ in the conclusion of Theorem \ref{thm-norma-nilp-same-tdlc-env}. 
\end{Remark}

\begin{Remark}
	As explained in the introduction, the assumption of Theorem \ref{thm-norma-nilp-same-tdlc-env} is equivalent to asking that there is a connected locally finite graph on which $\Gamma$ and $\Lambda$  act faithfully and geometrically. Since the conclusion of the theorem (i.e.\ the property of having a finitely generated normal subgroup of a given virtual isomorphism class), is stable under modding out, and forming an extension by, a finite normal subgroup, it  follows that the statement still holds if one requires that there is a connected locally finite graph on which $\Gamma$ and $\Lambda$  act geometrically (dropping the faithfulness assumption).
\end{Remark}

\subsection{Other rigidity results} \label{subsec-related-rigidity-results}

 Here we connect the setting of Theorem \ref{thm-norma-nilp-same-tdlc-env} with other results from the literature. Mosher--Sageev-Whyte's Theorem 2 in \cite{Mosher-Sageev-Whyte} shows that if a group $\Gamma$ acts cocompactly on an infinitely ended locally finite tree such that vertex stabilizers (which are commensurated subgroups of $\Gamma$) are finitely generated and nilpotent, and if $\Lambda$ is a group QI to $\Gamma$, then $\Lambda$ acts cocompactly on an infinitely ended locally finite tree with vertex stabilizers QI to those of $\Gamma$. That result has been generalized by Margolis \cite{Margolis-almost-normal}. The setting of \cite[Theorem 1.4]{Margolis-almost-normal} covers the situation of a group $\Gamma$ having a finitely generated nilpotent subgroup $A$ that is commensurated in $\Gamma$, and provides sufficient conditions under which every group $\Lambda$ that is QI to $\Gamma$  admits a finitely generated nilpotent subgroup $B$ such that $B$ is commensurated in $\Lambda$ and $B$ is QI to $A$. So \cite[Theorem 2]{Mosher-Sageev-Whyte} and \cite[Theorem 1.4]{Margolis-almost-normal} both consider commensurated subgroups, while Theorem \ref{thm-norma-nilp-same-tdlc-env} deals with normal subgroups (both in the assumption and in the conclusion). These two results hold in the more general setting of QI groups $\Gamma,\Lambda$, as opposed to the stronger assumption in Theorem \ref{thm-norma-nilp-same-tdlc-env}  that $\Gamma,\Lambda$ share a cocompact tdlc envelope. \cite[Theorem 1.4]{Margolis-almost-normal} requires on the one hand $\Gamma$ to be of type $F_n$ for a certain $n \geq 2$ (depending on $A$), and on the other hand that the coset space $\Gamma / A$ has infinitely many ends (the case where $\Gamma / A$ is a tree corresponding to  \cite[Theorem 2]{Mosher-Sageev-Whyte}). Theorem \ref{thm-norma-nilp-same-tdlc-env} has no such assumption.

\subsection{Almost self-centralizing nilpotent normal subgroups}

\begin{Definition}
We say that a subgroup $A$ of a group $\Gamma$ is self-centralizing in $\Gamma$ if $C_\Gamma(A)$ is contained in $A$, and  almost self-centralizing in $\Gamma$ if $C_\Gamma(A)$ has a finite index subgroup contained in $A$. 
\end{Definition}

So self-centralizing means $Z(A) =  C_\Gamma(A)$, and almost self-centralizing means $Z(A)$ has finite index in $C_\Gamma(A)$. 

\begin{Lemma} \label{lem-self-centr-fi}
If $A$ is normal in $\Gamma$ and $A$ is almost self-centralizing in $\Gamma$, then every finite index subgroup $A'$ of $A$ is almost self-centralizing in $\Gamma$.
\end{Lemma}

\begin{proof}
Since $A$ is normal, $C_\Gamma(A')$ acts by conjugation on $A$, and hence acts on the quotient $A/A'$. Since $A/A'$ is finite, $C_\Gamma(A')$ has a finite index subgroup that acts trivially on $A/A'$, so $C_\Gamma(A)$ is of finite index in $C_\Gamma(A')$. We deduce $Z(A)$ is of finite index in $C_\Gamma(A')$. Since $Z(A)$ is  virtually contained in $Z(A')$, $Z(A')$ is of finite index in $C_\Gamma(A')$. 
\end{proof}

\begin{Theorem} \label{thm-gen-envelope-normal-nilp}
Let $\Gamma$ be a group admitting a normal subgroup $A \lhd \Gamma$ such that $A$ is finitely generated and nilpotent, and $A$ is almost self-centralizing in $\Gamma$. If $G$ is a tdlc cocompact envelope of  $\Gamma$, then $G$ is compact-by-discrete. 
\end{Theorem}

\begin{proof}
By Lemma  \ref{lem-self-centr-fi}, upon changing $A$ into a finite index subgroup we can assume $A$ is torsion-free. According to Theorem \ref{thm-abelien-norm-co}, there exists a compact open subgroup $U$ of $G$ such $A$ normalizes $U$. We can therefore apply Proposition \ref{prop-commens-upgrade-normal}, which provides a finite index subgroup $G'$ of $G$ containing $\Gamma$, and a closed normal subgroup $M$ of $G'$ such that $A \leq M$ and $A$ is cocompact in $M$. If the conclusion holds for $G'$ then it also holds for $G$, so we can assume $G'=G$. By Theorem  \ref{thm-tdlc-poly-growth} we infer that $W(M)$ is a compact open subgroup of $M$.  Since $W(M)$ is characteristic in $M$ and $M$ is normal in $G$, $W(M)$ is normal in $G$. Since the conclusion is invariant under modding out $G$ by a compact normal subgroup, and since $A \cap W(M)$ is trivial since $A$ is torsion-free, we can assume $W(M)$ is trivial. It then follows that $M$ is discrete and finitely generated. By Lemma \ref{lem-QZ-open-centr} we deduce that $N := C_G(M)$ is an open normal subgroup of $G$. Hence $\Gamma \cap N$ is a discrete and cocompact subgroup of $N$. We have $\Gamma \cap N \leq C_\Gamma(A \cap M)$. Since $A \cap M$ has finite index in $A$, $C_\Gamma(A)$ has finite index in $C_\Gamma(A \cap M)$. So $\Gamma \cap N$ is virtually contained in $C_\Gamma(A)$, and hence in $Z(A)$ by the assumption that $A$ is almost self-centralizing in $\Gamma$. Hence  $\Gamma \cap N$ is virtually abelian. By Theorem  \ref{thm-tdlc-poly-growth}  again we deduce that $W(N)$ is compact and open in $N$. Since $N$ is open in $G$ it follows that $W(N)$ is a compact open normal subgroup of $G$. 
\end{proof}

\begin{Corollary}
Suppose $\Gamma$ is a group of the form $\Gamma = \Z^n \rtimes \Delta$, where $n \geq 2$ and $\Delta$ is any subgroup of $\mathrm{GL}(n,\Z)$. Then every tdlc cocompact envelope of  $\Gamma$ is compact-by-discrete. 
\end{Corollary}

\begin{proof}
By definition $A =  \Z^n$ is self-centralizing in $\Gamma$, so $\Gamma$ falls under the scope of Theorem \ref{thm-gen-envelope-normal-nilp}.
\end{proof}

 \section{Polycyclic groups} \label{sec-polycyclic-commable}

\begin{Theorem} \label{thm-tdlc-env-poly}
If $G$ is a tdlc cocompact envelope of a polycyclic group $\Gamma$, then $G$ is compact-by-discrete. 
\end{Theorem}

\begin{proof}
 Let $A$ denote the Fitting subgroup of $\Gamma$. Then $A$ is finitely generated nilpotent, and $A$ is self-centralizing in $\Gamma$ \cite[1.2.10]{Lennox-Rob}. The statement hence follows from Theorem \ref{thm-gen-envelope-normal-nilp}. 
\end{proof}

\begin{Remark}
	It is a general fact that if a group $\Gamma$ has an upper bound on the cardinality of its finite subgroups, then every tdlc envelope of $\Gamma$ is cocompact (see e.g.\ \cite[Lemma 3.9]{BFS-envelopes}). Since polycyclic groups are virtually torsion-free, they do satisfy this property. Hence Theorem \ref{thm-tdlc-env-poly} accounts for all tdlc envelopes of $\Gamma$. 
\end{Remark}

\begin{Remark}
It could also be possible to prove Theorem \ref{thm-tdlc-env-poly} with the approach taken in Section \ref{sec-solv-finite rank}. However the above proof relying on the results of Section \ref{sec-general-result-tdlc} is more self-contained. 
\end{Remark}
 
  The following is a reformulation of \cite[3.1]{Kropholler-90}, also proven in \cite[Proposition 5.8]{BCGM-amen}.
  
 \begin{Proposition} \label{prop-dense-poly-SIN}
 	Let $G$ be a tdlc group, and assume that $G$ admits a dense polycyclic subgroup. Then $G$ is compact-by-discrete.
 \end{Proposition}

 \begin{Theorem} \label{thm-non-closed-cocomp-poly}
 	If $G$ is a tdlc group containing a cobounded polycyclic subgroup, then $G$ is compact-by-discrete. 
 \end{Theorem}
 
 \begin{proof}
 	Let $\Gamma$ be a cobounded polycyclic subgroup of $G$, and let $H = \overline{\Gamma}$. By Proposition \ref{prop-dense-poly-SIN}, $H$ admits a compact normal subgroup $K$ such that $H/K$ is discrete. Since $H$ is cocompact in $G$, by Proposition \ref{prop-polycompact-cocompact} the compact subgroup $K$ is contained in a compact subgroup $K'$ such that $K'$ is normal in $G$. The image of $\Gamma$ in $G/K'$ is therefore discrete and cocompact. By Theorem \ref{thm-tdlc-env-poly} we deduce $G/K'$  is compact-by-discrete. Therefore $G$ as well. 
 \end{proof}

 \subsection{Stability under commability} \label{subsec-stable-commable}
 
  Two locally compact (e.g.\ discrete) groups $G,H$ are {commable} if there exist locally compact groups $G_0, \ldots, G_n$ with $G_0 = G$  and $G_n = H$ and homomorphisms $G_0 \to G_1 \leftarrow G_2 \to \cdots \to G_{n-1} \leftarrow G_n$, where each $G_i \to G_{i+1}$ and $G_{i+1} \leftarrow G_{i+2}$ represents a continuous homomorphism with compact kernel and closed cocompact image \cite{Cor-comm-focal}. This can also be phrased in terms of geometric actions. Two $\sigma$-compact groups $G,H$ that are capable of acting continuously and geometrically on the same proper metric space are commable, and for $\sigma$-compact groups commability is the equivalence relation generated by this property. The equivalence between these points of view is explained in details in \cite[4.C-5.B]{CorHar}. When $G,H$ are compactly generated, the (locally compact version of the) Milnor-Schwarz lemma asserts that commable implies QI.

 The study of all cocompact envelopes of a group $\Gamma$ contributes to the more general problem of the study of the locally compact groups commable to $\Gamma$, as it consists in the first step $\Gamma \to G_1$ in the above sequence (up to the fact that we also allow to quotient out by a finite normal subgroup). 
 
 \begin{Theorem} \label{thm-commable-rigidity}
 	Consider the class of locally compact groups $G$ such that, after modding out by a compact normal subgroup, we have: \begin{enumerate}[label=(\arabic*)]
 		\item \label{item-unimod-ame} $G$ is unimodular and amenable;
 		\item \label{item-Lie-by-poly} the identity component $G^0$ is open in $G$, and $G/G^0$ is virtually polycyclic.
 	\end{enumerate}
Then this class is stable under commability. 
 \end{Theorem}

A tdlc group belongs to the class described in the theorem if and only if it is compact-by-(discrete virtually polycyclic). Hence Theorem \ref{thm-commable-rigidity} generalizes Theorem \ref{thm-tdlc-env-poly}. A discrete group belongs to that class if and only if it is virtually polycyclic. Hence we have:

\begin{Corollary}
	Let $\Gamma$ be a virtually polycyclic group, and let $\Lambda$ be a discrete group that is commable to  $\Gamma$. Then $\Lambda$ is virtually polycyclic. 
\end{Corollary}

Recall that a connected Lie group $G$ is amenable if and only if $G$ has a closed normal cocompact solvable subgroup \cite[Theorem 1.7]{Furst-63-Poisson-formula}. Recall also that if $H$ is a closed subgroup of a connected amenable Lie group $G$, then $H/H^0$ is virtually polycyclic (in particular discrete subgroups of $G$ are virtually polycyclic).

\begin{proof}[Proof of Theorem \ref{thm-commable-rigidity}]
Let $\mathcal{C}_1$ denote the class of locally compact groups that satisfy properties \ref{item-unimod-ame} and \ref{item-Lie-by-poly} of Theorem \ref{thm-commable-rigidity}. Let $\mathcal{C}_2$ be the class of groups $G$ such that $G$ admits a compact normal subgroup $K$ such that $G/K$ is in $\mathcal{C}_1$.	We have to show $\mathcal{C}_2$ is stable under commability. Clearly $\mathcal{C}_2$ is stable under forming an extension by a compact normal subgroup and modding out by a compact normal subgroup. So what we have to prove is that if $H$ is a closed cocompact subgroup of $G$ and one of $H$ or $G$ is in $\mathcal{C}_2$, then so is the other. 

Suppose first $G \in \mathcal{C}_2$. We want to show $H \in \mathcal{C}_2$. After modding out by a compact subgroup we are reduced to show that if $G \in \mathcal{C}_1$, then $H \in \mathcal{C}_1$. Also by Theorem \ref{thm-Hilbert-fifth}, we can assume $G^0$ is a connected Lie group. Since property \ref{item-unimod-ame} is a QI-invariant among compactly generated locally compact groups \cite[Corollary 11.13]{Tessera-LS-Sobolev}, $H$ satisfies \ref{item-unimod-ame}. Now let $L := H \cap G^0$. Since $G^0$ is open in $G$ and $H$ is closed in $G$, the subgroup $L$ is open in $H$. Since $L$ is a closed subgroup of a connected Lie group, $L^0$ is open in $L$. Since $H^0 = L^0$, $H^0$ is open in $H$. We shall see that $H/H^0$ is virtually polycyclic. Since $G/G^0$ is virtually polycyclic and this property is stable by extension, it is enough to see that $(H \cap G^0) / H^0$ is virtually polycyclic. But $H \cap G^0$ is a closed subgroup of an amenable connected Lie group, and hence the quotient by its connected component of the identity (which is $H^0$), is indeed a virtually polycyclic group.

Suppose now $H \in \mathcal{C}_2$. Again by \cite[Corollary 11.13]{Tessera-LS-Sobolev}, $G$ satisfies \ref{item-unimod-ame}. Let $Q$ be the totally disconnected group $Q  = G / G^0$, and denote by $\Gamma$ the image of $H$ in $Q$. Since $H^0 \leq G^0$ and $H$ verifies \ref{item-Lie-by-poly}, $\Gamma$ is virtually polycyclic. Since $H$ is cocompact in $G$, $\Gamma$ is cobounded in $Q$. So we can apply Theorem \ref{thm-non-closed-cocomp-poly}, from which we deduce that $Q$ is compact-by-discrete. Therefore $G$ is (connected-by-compact)-by-discrete. According to Theorem \ref{thm-Hilbert-fifth}, $G$ is  (compact-by-(virtually connected Lie)-by-discrete. The normal compact-by-(virtually connected Lie) subgroup of $G$ has a unique maximal compact normal subgroup (Proposition \ref{prop-Lie-WG-cpct}), which is therefore normal in $G$. Hence after modding out by this compact subgroup the group $G$ is (connected Lie)-by-discrete. The image of $H$ in $G/G^0$ is virtually polycyclic and has finite index in $G/G^0$. So $G/G^0$ is indeed virtually polycyclic, i.e.\ $G$ verifies  \ref{item-Lie-by-poly}. 
\end{proof}

\subsection{Application: classification of all envelopes of certain polycyclic groups}

In this subsection we prove Theorem \ref{thm-intro-classification-env-polyc} from the introduction.

 Let $d \geq 1$. We say that an element of $\mathrm{GL}(d,\R)$ is generic if it has $d$ distinct real eigenvalues. Let $A \leq \mathrm{GL}(d,\Z)$ be a free abelian group of rank $k \geq 1$. Suppose that every non-trivial element of $A$ is generic and has positive eigenvalues. So $A$ is diagonalizable over $\R$, and one can find commuting real matrices $X_1,\ldots,X_k $ such that, if $\varphi : \R^k \to \mathrm{GL}(d,\R)$ is the the  homomorphism defined by $\varphi(t_1,\ldots,t_k) = \exp(t_1 X_1 + \cdots t_k X_k)$, then $\varphi(\Z^k) = A$.

We have the following basic lemma: 

\begin{Lemma} \label{lem-unique-torus}
	If a continuous homomorphism $\psi : \R^k \to \mathrm{GL}(d,\R)$ coincides with $\varphi$ on $\Z^k$, then $\psi = \varphi$. 
\end{Lemma}

\begin{proof}
	By restricting to each coordinate, it is enough to consider the case $k=1$. That case amounts to see that if $A = \exp(X)$ is generic, then $X$ is uniquely determined by $A$. Let $(v_1,\ldots,v_d)$ be a basis made of eigenvectors of $A$. Since $A$ is diagonalizable, so is $X$. If $(w_1,\ldots,w_d)$ is  a basis made of eigenvectors of $X$, then each $w_i$ is also an eigenvector of $A$. Since eigenspaces of $A$ are lines, it follows that up to a permutation we have $(w_1,\ldots,w_d) = (v_1,\ldots,v_d)$. Hence $X$ acts on $v_i$ by multiplication by $\ln(\lambda_i)$, and $X$ is indeed uniquely determined. 
\end{proof}

In the sequel we denote  $N = \Z^d $ and $\Gamma = N \rtimes A$. The group $\Gamma$ therefore sits as a discrete and cocompact subgroup in $G_{\varphi} =  \R^d \rtimes_\varphi \R^k$. We will make the following standing assumption:  \begin{equation}
\label{eq:stand-ass}
\text{Every element of $A$ is generic with eigenvalues positive and  different from $1$.} \tag{$\dagger$}
\end{equation}

\begin{Proposition} \label{prop-Lie-env-Gphi}
Suppose $(\dagger)$. Suppose $G$ is a connected Lie group that is an envelope of $\Gamma$. Then, after modding out by a compact normal subgroup, $G$ is isomorphic to $G_{\varphi}$.   
\end{Proposition}

\begin{proof}
	Upon modding out by the unique maximal compact normal subgroup, we assume $G$ has no non-trivial compact normal subgroup. Suppose for a moment that the case where $G$ is solvable has been treated. Let $R$ be the largest normal closed connected solvable subgroup of $G$. We know that $R$ is cocompact in $G$ since $G$ is amenable. Since $R$ is normal, $L := \overline{R \Gamma}$ is a closed cocompact subgroup of $G$. Since both $\Gamma$ and $R$ are solvable, $L$ is solvable.  Let $L^0$ be the identity component of $L$. We have $R \leq L^0$, and $L^0$ is a finite index subgroup of $L$. Since the statement is invariant under changing $\Gamma$ by a finite index subgroup, we can assume $L^0$ contains $\Gamma$. Hence by the assumption we have that $L^0$ is isomorphic to $G_{\varphi}$. In particular $L^0$ is simply connected. Since $R$ is a closed connected cocompact subgroup, we deduce $R = L^0$. So we are reduced to show: whenever a connected Lie group $G$ contains $G_{\varphi}$ as a cocompact normal subgroup, we have $G = G_{\varphi}$. For, let $\alpha: G \to  \mathrm{GL}(d,\R)$ the representation associated to the $G$-action on the normal subgroup $\R^d$. One can write $G = G_{\varphi} K$ with $K$ compact maximal. Since $\R^k$ contains generic elements, the $\alpha(\R^k)$-action on the projective space $\mathrm{P}(\R^d)$ has exactly $d$ fixed points. Since $\alpha(\R^k)$ is normal in $\alpha(\R^k K)$,  $\alpha(K)$ preserves this finite set of fixed points. Since $K$ is connected and compact, $\alpha(K)$ is trivial. So $K$ centralizes $\R^d$, and $C_G(\R^d) = K \R^d$. Since $\R^d$ is normal in $G$, so is $C_G(\R^d)$. In particular $K$ is normal in $G$, and hence trivial. So $G = G_{\varphi}$. 
	
	We now assume $G$ is solvable. Let $N_G$ denote the largest normal closed connected nilpotent subgroup. The quotient group $G/N_G$ is abelian, and $N_G$ is simply connected since $G$ has no non-trivial compact normal subgroup. By a theorem of Mostow \cite[Theorem 3.3]{Raghu}, the intersection $\Gamma \cap N_G$ is a discrete and cocompact subgroup of $N_G$. Since elements of $A$ have no eigenvalues equal to $1$, there can be no non-zero factor of $N \otimes \Q$ on which the $A$-action is unipotent. Hence the derived subgroup $[\Gamma,\Gamma]$ has finite index in $N$. A fortiori $\Gamma \cap N_G$  has finite index in $N$. So $N_G \simeq \R^d$ and the $A$-action by conjugation on $N_G$ is the original action.  We now want to see $G$ is simply connected. Let $K$ be a maximal compact subgroup of $G$. The subgroup $R = N_G K$ is normal in $G$ because $G/N_G$ is abelian, so the image of $K$ in $ \mathrm{GL}(d,\R)$ is a compact connected subgroup that is normalized by $A$. Since $A$ contains a generic element, this subgroup must be trivial. So $K$ centralizes $N_G$. In particular $K$ is normal in $R$, and hence in $G$. So $K=1$, $G$ is simply connected, and $G/N_G$ is isomorphic to $\R^k$. By Lemma  \ref{lem-unique-torus} the $G/N_G$-action on $N_G$ must be isomorphic to the given one. Moreover the sequence $1 \to \R^d \to G \to \R^k \to 1$ necessarily splits (see e.g.\ \cite[Lemma 5.3.10]{Peng-II}), so $G \simeq G_\varphi$.
\end{proof}

\begin{Remark}
	In the second part of the proof, where $G$ is solvable, if we assume in addition $G$ is simply connected and the Zariski closure of the adjoint representation of $G$ has no non-trivial compact torus, then the result is a consequence of a much more general result of Witte  \cite[Theorem 7.2]{Witte-superrigid}. But verifying the above additional assumptions seems to require essentially the same arguments than the ones above.
\end{Remark}

\begin{Theorem} \label{thm-classification-env}
Suppose $(\dagger)$. Suppose $G$ is a locally compact group that is an envelope of $\Gamma = N \rtimes A \simeq  \Z^d \rtimes \Z^k$. Then, up to compact normal subgroup and finite index subgroup, $G$ embeds as a closed cocompact subgroup in $G_{\varphi} =  \R^d \rtimes_\varphi \R^k$.   
\end{Theorem}

\begin{proof}
	According to Theorem \ref{thm-commable-rigidity}, and after modding out by compact normal, we have that $G^0$ is open in $G$. By Theorem \ref{thm-Hilbert-fifth} we can assume $G^0$ is a connected Lie group. Also by Proposition \ref{prop-Lie-WG-cpct} we can assume $G^0$ has no non-trivial compact normal subgroup. In particular upon passing to a finite index subgroup we can assume $G = G^0 \Gamma$. The subgroup $\Lambda := \Gamma \cap G^0$ is discrete and cocompact in $G^0$. If $G^0$ is compact then it is trivial and $G=\Gamma$. Assume $G^0$ is not compact. Then $\Lambda$ is an infinite normal subgroup of $\Gamma$. We now distinguish two cases. Suppose first that $\Lambda \subseteq N$. Then $\Lambda$ is free abelian of rank $r \geq 1$. By Bieberbach theorem one can write $G^0 = R \rtimes K$ with $K$ compact, $\Lambda$ virtually contained in $R$, and $R$ is isomorphic to $\R^r$ and is identified with a $A$-invariant subspace of $\R^d$. For the same reason as in the proof of Proposition \ref{prop-Lie-env-Gphi}, we infer $K$ must be trivial. We deduce $G = R \Gamma = (RN) \rtimes A$. The subgroup $N$ centralizes a lattice of $R$, so $N$ and $R$ commute. So $RN$ is abelian. Since $RN$ has no compact normal subgroup, we deduce $RN$ is connected, and $RN = R$. So $G \simeq ( \R^r \times \Z^{d-r}) \rtimes A$, and $G$ is indeed isomorphic to a closed cocompact subgroup of $G_{\varphi}$. Suppose now $\Lambda$ is not contained in $N$. Then by Lemma  \ref{lem-normal-Gamma-similar} below,  up to passing to finite index subgroup, $\Lambda \simeq \Z^d \rtimes \Z^\ell$ with $1 \leq \ell \leq k$, where the semi-direct product verifies the same properties as $\Gamma$.  Hence Proposition  \ref{prop-Lie-env-Gphi} applies to $\Lambda$, and we deduce $G^0 \simeq \R^d \rtimes \R^\ell$, where  the action of $\R^\ell$ is via the restriction of $\varphi$. As in the previous case we have $N \leq \R^d$, and consequently $G = \R^d  \rtimes (\R^\ell A)$.
\end{proof}

We have used: 

\begin{Lemma} \label{lem-normal-Gamma-similar}
Suppose $(\dagger)$. Suppose that $\Lambda$ is normal in $\Gamma = N \rtimes A $, and $\Lambda$ is not contained in $N$. Then $\Lambda \cap N \simeq \Z^d$,  $\Lambda \cap A \simeq \Z^\ell$ for some $\ell \geq 1$ and $(\Lambda \cap N) \rtimes (\Lambda \cap A)$ has finite index in $\Lambda$. 
\end{Lemma}

\begin{proof}
	The largest normal nilpotent  subgroup $N_\Lambda$ of $\Lambda$ is normal in $\Gamma$, and hence contained in $N$. The quotient $\Lambda / N_\Lambda$ is abelian, and is a normal subgroup of $N / N_\Lambda \rtimes A$. Suppose $V = N / N_\Lambda \otimes \Q$ has positive dimension. Then the $A$-action on $V$ keeps the property that non-trivial elements are generic with eigenvalues different from $1$. This forces an abelian normal subgroup of $N / N_\Lambda \rtimes A$ to be contained in $N / N_\Lambda$. Hence $\Lambda \leq N$, a contradiction. So $N / N_\Lambda$ is finite, and upon changing $\Gamma$ into the finite index subgroup $\Gamma' = N_\Lambda \rtimes A$ and $\Lambda$ into $\Lambda \cap \Gamma'$, we can assume $N_\Lambda = N$. We then have $\Lambda = N \rtimes (\Lambda \cap A)$, and $\Lambda \cap A$ is infinite, as desired. 
\end{proof}

\begin{Remark}
Retain the setting of Theorem \ref{thm-classification-env}. Let $F \leq \mathrm{GL}(d,\R)$ be a maximal finite subgroup of the normalizer of $\varphi(\R^k)$ in $\mathrm{GL}(d,\R)$, and let $G_{\varphi}' := G_{\varphi} \rtimes F$. It is possible that one could avoid passing to a finite index subgroup of $G$ in Theorem \ref{thm-classification-env} provided $G_{\varphi}$ is replaced by $G_{\varphi}'$ in the conclusion. In other words, it seems plausible that every locally compact group that is an envelope of $\Gamma$ embeds, up to compact normal subgroup, as a closed cocompact subgroup in $G_{\varphi}'$. 
\end{Remark}

\subsection{Tdlc envelopes of lattices in  Lie groups} \label{sec-latt-Lie}

Let $H$ be a connected Lie group, and $\Gamma$ a lattice in $H$. Furman and Bader--Furman--Sauer showed that if $H$ is semi-simple and not isomorphic to $\mathrm{PSL}_2(\R)$, then every tdlc envelope of $\Gamma$ is compact-by-discrete \cite{Furman-Mostow-LC, BFS-envelopes}. The requirement to exclude $\mathrm{PSL}_2(\R)$ is necessary, see the discussion above Theorem C in \cite{BFS-envelopes}. At the opposite extreme, if $H$ is amenable, then the group $\Gamma$ is virtually polycyclic. In that case by Theorem \ref{thm-tdlc-env-poly} it is again the case that every tdlc envelope of $\Gamma$ is compact-by-discrete. The following result establishes that the same conclusion still holds even if $H$ is not assumed to be either semi-simple or amenable.

If $H$ is a connected Lie group, we denote by $R(H)$ the largest amenable normal subgroup of $H$. The quotient $H / R(H)$ is a connected semi-simple Lie group. 

\begin{Theorem} \label{thm-env-lat-general-Lie}
	Let $H$ be a connected Lie group such that $H / R(H)$ is not isomorphic to $\mathrm{PSL}_2(\R)$. Let $\Gamma$ be a lattice in $H$. Then every tdlc envelope of $\Gamma$ is compact-by-discrete. 
\end{Theorem}

\begin{proof}
	Let $G$ be a tdlc envelope of $\Gamma$.  We shall first argue that $G$ is a cocompact envelope. As recalled earlier, it is enough to see that finite subgroups of $\Gamma$ have bounded cardinality. The image of $\Gamma$ under the adjoint representation of $H$ is virtually torsion-free by Selberg's lemma, so in particular its finite subgroups have bounded cardinality. Finite subgroups of $\Gamma \cap Z(H)$ also have bounded cardinality. Since the property of having finite subgroups of bounded cardinality is stable under extensions, $\Gamma$ indeed has this property.
	
	Set $P := \Gamma \cap 	R(H)$. By Auslander theorem, $P$ is discrete and cocompact in $R(H)$, and the image $\Gamma / P$ of $\Gamma$ in the semi-simple group $H / R(H)$ is a lattice in $H / R(H)$  \cite[Corollary 8.25-8.27]{Raghu}. The subgroup $P$ is virtually polycyclic, and is normal in $\Gamma$. Let $A$ be the largest nilpotent normal subgroup of $P$. The subgroup $A$ being characteristic in $P$, $A$ is normalized by $\Gamma$. Since $G$ is a cocompact envelope of $\Gamma$, we are in position to apply Theorem  \ref{thm-abelien-norm-co}. We deduce that $A$ normalizes a compact open subgroup of $G$. 
	By Proposition \ref{prop-commens-upgrade-normal} applied to $(A,\Gamma,G)$, we obtain (up to passing to finite index subgroup of $G$) a closed normal subgroup $M$ of $G$ such that $A \leq M$ and $A$ is cobounded in $M$. Let $A'$  be the image of $P$ in $G/M$. The subgroup $A'$ is virtually abelian, and $A'$ is normalized by the image of $\Gamma$ in $G/M$, which is a cobounded subgroup of in $G/M$. We can therefore apply Theorem  \ref{thm-abelien-norm-co} in the ambient group $G/M$. This provides a compact open subgroup $U$ of $G/M$ such that $A'$ normalizes $U$. Let $O$ be the pre-image in $G$ of $A' U$. Then $O$ is an open subgroup of $G$ containing $P$ and such that $P$ is cocompact in $O$. Theorem \ref{thm-tdlc-env-poly} then asserts $P$ normalizes a compact open subgroup of $O$ (and hence of $G$). Proposition \ref{prop-commens-upgrade-normal} applied to $(P,\Gamma,G)$ ensures $G$ (virtually) has a closed normal subgroup $N$ such that $P \leq N$ and $P$ is cocompact in $N$. This subgroup $N$ is necessarily compact-by-discrete. Since $P$ is cocompact in $N$, the image of $\Gamma$ in $G/N$ is discrete and cocompact in $G/N$. So  the group $G/M$ is a cocompact envelope of $\Gamma / P$. Since $\Gamma / P$ is a lattice in the semi-simple group $H/R(H)$, and $H/R(H)$ is not isomorphic to $\mathrm{PSL}_2(\R)$, by \cite{Furman-Mostow-LC,BFS-envelopes} the group $G/N$ is compact-by-discrete. So all together we have that $N$ has a unique maximal compact open normal subgroup, $N$ is compactly generated, and $G/N$ is compact-by-discrete. By \cite[Lemma 4.8]{BCGM-amen} this implies $G$ is compact-by-discrete. 
\end{proof}

\section{Solvable groups of finite rank} \label{sec-solv-finite rank}

The main goal of that section is to prove Theorem  \ref{thm-intro-Rad-Lambda-fin} from the introduction. It will be convenient to work in the class of minimax groups.

For a prime $p$, the Prüfer $p$-group is the quotient $\Z[1/p] / \Z$. A solvable group $\Gamma$ is minimax if it admits a series in which each factor is either cyclic of a Prüfer $p$-group. The number $h(\Gamma) \geq 0$ of infinite cyclic factors in a defining series of $\Gamma$ is the Hirsch number of $\Gamma$. If a group $\Gamma$ is virtually minimax, we extend this by defining $h(\Gamma)$ to be the common Hirsch number of all minimax finite index subgroups of $\Gamma$. For $\Gamma$  virtually minimax, we denote by $R(\Gamma)$ the finite residual of $\Gamma$, i.e.\ the intersection of all finite index subgroups of $\Gamma$. 

\begin{Proposition} \label{prop-background-minimax}
	The following hold:
		\begin{enumerate}
			\item Every minimax group is of finite rank. Conversely, every finitely generated solvable group of finite rank is minimax. Hence for finitely generated solvable groups, minimax and finite rank is the same (\cite[\S 5.1]{Lennox-Rob}).
			\item 	Let $\Gamma$ be a minimax group. Then:
	\begin{enumerate}
		\item The Fitting subgroup  $\Fit(\Gamma)$ of $\Gamma$ is nilpotent, and the quotient $\Gamma / \Fit(\Gamma)$ is virtually abelian. If moreover $\Gamma$ is virtually torsion-free, $\Gamma / \Fit(\Gamma)$ is finitely generated (\cite[5.2.3]{Lennox-Rob}). 
		\item  \label{item-minimax-RGamma} The group $R(\Gamma)$  is a direct product of finitely many Prüfer groups, $R(\Gamma) \leq \Fit(\Gamma)$, and $R(\Gamma) \leq W(\Gamma)$ (\cite[1.4.1--5.2.1--5.3.2]{Lennox-Rob}).
		\item   The group $\Gamma$ is virtually torsion free if and only if $R(\Gamma)$ is trivial, if and only if $W(\Gamma)$ is finite (\cite[5.2.5]{Lennox-Rob}). 
		\item \label{item-minimax-Hirsch-zero} $h(\Gamma) = 0$ if and only if $R(\Gamma)$ has finite index in $\Gamma$ (\cite[\S 1.4]{Lennox-Rob}). 
	\end{enumerate}
	\end{enumerate}
\end{Proposition}

\subsection{Envelopes of solvable groups of finite rank}

The goal of this subsection is to prove the following result about the structure of cocompact envelopes of virtually minimax groups:

\begin{Theorem} \label{thm-gen-structure-envelopes}
	Let $\Gamma$ be a virtually minimax group, and $G$ a cocompact envelope of $\Gamma$. Then the subgroup $G^0 \, \RadLE(G)$ is open in $G$.
\end{Theorem}

For every locally compact group $G$, the intersection  $G^0 \cap \RadLE(G)$ is always compact. Hence in the theorem if in addition $G$ has no non-trivial compact normal subgroup, then  $G^0 \, \RadLE(G)$ is topologically isomorphic to  $G^0 \times \RadLE(G)$, and is open in $G$. 

\begin{Remark}
In the theorem none of $G^0$ or $\RadLE(G)$ can be removed. For instance the group $\Gamma = \Z[1/p] \rtimes_p \Z$ admits as a cocompact envelope the group $G = (\R \times \Q_p)  \rtimes_p \Z$, for which $G^0 = \R$ and $\RadLE(G) = \Q_p$. More generally, every finitely generated virtually torsion-free minimax group $\Gamma$  admits a cocompact envelope $G$ in which none of  $G^0$ or $\RadLE(G)$ is compact, provided $\Gamma$ is not virtually polycyclic \cite[\S 8.1]{Cor-Tess-vanishing}. 
\end{Remark}

\begin{Proposition} \label{prop-limit-radical-by-abelian}	
	Let $G = \bigcup_i G_i$ be a locally compact group written as a directed union of open subgroups $G_i$. Let $R_i = \RadLE(G_i)$, and suppose that $G_i /R_i$ is virtually minimax of Hirsch number $d_i$. If $(d_i)$ is bounded, then the following hold: \begin{enumerate}
		\item $(d_i)$ is eventually constant equal to $d = \max d_i$;
		\item $R := \RadLE(G)$ is open in $G$, and $ R = \bigcup_i R_i$;
		\item  $G/ R$ is minimax and $h(G/ R) = d$.
	\end{enumerate}
\end{Proposition}

\begin{proof}
	The key is to see that $R_i \subseteq R_j$ whenever $j \geq i$. Fix $i,j$ such that $j \geq i$. Let $\pi_i$ be the quotient map from $G_i$ to $G_i /R_i$. The subgroup $R_j \cap G_i$ is normal in $G_i$ and locally elliptic, hence contained in $R_i$. So we have a short exact sequence \[ 1 \to R_i / R_j \cap G_i \to G_i / R_j \cap G_i \to G_i / R_i \to 1.\] Observe that since  $\RadLE(G_j / R_j)$ is trivial, Proposition   \ref{prop-background-minimax} ensures $G_j /R_j$ is virtually torsion-free. Hence $G_i / R_j \cap G_i$, which is a subgroup of $G_j /R_j$, is virtually torsion-free as well. Hence $R_i / R_j \cap G_i $ is finite, and $G_i / R_j \cap G_i$ is virtually minimax  of Hirsch number $d_i$. Since $G_i / R_j \cap G_i$ is a subgroup of $G_j / R_j$, we deduce $d_i \leq d_j$. The assumption that $(d_i)$ is bounded hence means that $(d_i)$ is eventually constant, and without loss of generality we can assume $d_i = d$ for every $i$. We deduce in particular $G_i / R_j \cap G_i$  has finite index in $G_j / R_j$. Since $G_j /R_j$ has trivial locally elliptic radical, the same is true for every finite index subgroup. Hence $R_i / R_j \cap G_i $ is actually trivial, and hence $R_i = R_j \cap G_i$. In particular $R' := \bigcup_i R_i$ is a subgroup of $G$. It is locally elliptic, open and normal. Moreover the locally elliptic radical of $G$ has the property that $\RadLE(G) \cap G_i \leq R_i$, so $\RadLE(G) \leq R'$. Hence $\RadLE(G) = R'$.
\end{proof}

The proof of Theorem  \ref{thm-gen-structure-envelopes} will involve property $H_{FD}$. A locally compact group $G$ has property $H_{FD}$ if every unitary representation $\pi$ with non-zero first reduced cohomology $\bar{H}^1(G,\pi)$, has a sub-representation of positive finite dimension. We refer the reader to Shalom's original article \cite{Shalom-Acta} for background. We will be using the following result of Shalom:

\begin{Theorem} \label{thm-Shalom-virt-onto-Z}
Let $G$ be a non-compact, compactly generated and amenable tdlc group. If $G$ has property $H_{FD}$, then $G$ has a finite index open subgroup that surjects onto $\Z$. 
\end{Theorem}

\begin{proof}
This is Theorem 4.3.1 in \cite{Shalom-Acta}. The result is stated there for discrete groups, but the same arguments cover the above setting. Indeed, the first part of the proof of \cite[Theorem 4.3.1]{Shalom-Acta} provides a continuous homomorphism $\rho: G \to \mathbb{C}^n \rtimes U(n)$ with infinite image. Here since $G$ is totally disconnected and  $\mathbb{C}^n \rtimes U(n)$ has no small subgroups, $\ker(\rho)$ is necessarily open. Hence the image of $\rho$ is an infinite finitely generated subgroup of $\mathbb{C}^n \rtimes U(n)$, and the second part of the proof goes through.
\end{proof} 

We will also rely on the following result of Cornulier--Tessera, which is \cite[Theorem 1.12]{Cor-Tess-vanishing}. 

\begin{Theorem} \label{thm-Cor-Tess-Hfd-solv}
If $\Gamma$ is a finitely generated minimax group, then $\Gamma$ has property $H_{FD}$.
\end{Theorem}

Recall that a closed subgroup $H$ of a locally compact group $G$ has finite covolume if there is a $G$-invariant probability measure on $G/H$. 

%When $H$ is a closed cocompact subgroup of $G$ and $H$ is amenable,  $H$ has finite covolume in $G$ if and only if $G$ is amenable. 

\begin{Theorem} \label{thm-cobounded-tdlc-rad-open}
	Let $\Gamma$ be a subgroup of a tdlc group $G$ such that $\Gamma$ is cobounded in $G$ and $\overline{\Gamma}$ has finite covolume in $G$. If $\Gamma$ is virtually minimax, then $\RadLE(G)$ is open, and $G / \RadLE(G)$ is virtually minimax (with Hirsch number bounded above by the one of $\Gamma$). 
\end{Theorem}

\begin{proof}
First, we notice that	since $\overline{\Gamma}$ is amenable and has finite covolume in $G$, $G$ is amenable \cite[Corollary G.3.8]{BHV}. Second, we also notice that as soon as we have  proved that $\RadLE(G)$ is open, the remaining assertions follow because the discrete group $G / \RadLE(G)$ admits a quotient of $\Gamma$ as a finite index subgroup, and hence is indeed virtually minimax. The bound on the Hirsch number is also clear.
	
For, we argue by induction on $h(\Gamma)$. Suppose $h(\Gamma) = 0$. By items \ref{item-minimax-Hirsch-zero} and \ref{item-minimax-RGamma} of Proposition \ref{prop-background-minimax}, the subgroup $W(\Gamma)$ has finite index in $\Gamma$. Since $\Gamma$ is cobounded, Proposition \ref{prop-polycompact-cocompact} implies $W(\Gamma)$ is contained in $W(G)$. Since $W(G) \leq \RadLE(G)$, we infer that $\RadLE(G)$ is cocompact in $G$. But then $G / \RadLE(G)$ is both compact, and shall have trivial locally elliptic radical. So $\RadLE(G) = G$. 

Assume now $h(\Gamma) > 0$ and the results holds for every group of Hirsch number $< h(\Gamma)$. Write the group $G = \bigcup_i G_i$ as the directed union of its compactly generated open subgroups. Set $\Gamma_i = \Gamma \cap G_i$. The subgroup $\Gamma_i$ is cobounded in $G_i$, and $h(\Gamma_i ) \leq h(\Gamma)$ for all $i$. Hence if we show that $\RadLE(G_i)$ is open for every $i$, then Proposition \ref{prop-limit-radical-by-abelian} applies, and ensures that $\RadLE(G)$ is open. So we are reduced to show that $\RadLE(L)$ is open for every compactly generated open subgroup $L$ of $G$. Set $\Gamma_L := \Gamma \cap L$, which is a cobounded subgroup of $L$. Since $L$ is compactly generated, Proposition  \ref{prop-approx-fg} asserts that there is a sufficiently large finitely generated subgroup $\Lambda_L$ of $\Gamma_L$ such that $\Lambda_L$  is cobounded in $L$.  If $L$ is compact then there is nothing to prove, so we assume $L$ non-compact. Since $\Lambda_L$ is finitely generated,  Theorem \ref{thm-Cor-Tess-Hfd-solv} says that $\Lambda_L$  has property $H_{FD}$. Since $H_{FD}$ is inherited from a dense subgroup, we deduce that $\overline{\Lambda_L}$ also has $H_{FD}$. Also since $L$ is amenable, and $\overline{\Lambda_L}$ is cocompact in $L$, $\overline{\Lambda_L}$ is of finite covolume in $L$. Now $H_{FD}$ passes from a closed cocompact subgroup of finite covolume to the ambient group \cite[Proposition 4.13(1)]{Cor-Tess-vanishing}. Therefore $L$ has property $H_{FD}$. Since $L$ is compactly generated, amenable and non-compact, Theorem \ref{thm-Shalom-virt-onto-Z} ensures that $L$ has a finite index subgroup $L'$ that surjects onto $\Z$. Upon passing to a further finite index subgroup we can assume that $L'$ is normal, and then upon replacing $L$ by $L'$ we can assume $L'=L$. So we have a subgroup $N$ that is open and normal in $L$ such that $L/N \simeq \Z$. Hence $\Gamma \cap N$ is cobounded in $N$ and $h(\Gamma \cap N) < h(\Gamma)$. By induction $\RadLE(N)$ is open in $N$. Since $N$ is normal and the locally elliptic radical is characteristic, we  deduce $\RadLE(N) \leq \RadLE(L)$ and hence $ \RadLE(L)$ is open in $L$. 
\end{proof}

\begin{Proposition} \label{prop-G0-elliptic-open}
	Let $G$ be a locally compact group such that $G/G^0$ is locally elliptic. Then $G^0 \RadLE(G)$ is open in $G$.
\end{Proposition}

\begin{proof}
	See \cite[Theorem A.5]{CoTe-contract-Lp}. 
\end{proof}

\begin{proof}[Proof of Theorem \ref{thm-gen-structure-envelopes}]
	Since $\Gamma$ is amenable, $G$ is amenable. Hence so is the tdlc quotient $Q = G/G^0$. Applying Theorem  \ref{thm-cobounded-tdlc-rad-open} to  the image of $\Gamma$ in $Q$, we deduce $\RadLE(Q)$ is open in $Q$. Let $O$ be the preimage in $G$ of $\RadLE(Q)$. The subgroup $O$ is therefore open in $G$, and we have $O^0 = G^0$ and $\RadLE(O) = \RadLE(G)$. Moreover $O$ verifies the assumption of Proposition \ref{prop-G0-elliptic-open}, so by this proposition we infer $O^0 \RadLE(O)$ is open in $O$. Since $O$ is open in $G$ and $O^0 \RadLE(O) = G^0 \RadLE(G)$ , the statement follows.
\end{proof}

\begin{Corollary} \label{cor-tdlc-envelopes-minimax}
	If $G$ is a tdlc cocompact envelope of a virtually torsion-free minimax group $\Gamma$, then $G$ is compact-by-discrete. 
\end{Corollary}

\begin{proof}
Since $G$ is tdlc, Theorem  \ref{thm-gen-structure-envelopes} says $\RadLE(G)$ is open in $G$. But then $\Gamma \cap \RadLE(G)$ is torsion and cocompact in $\RadLE(G)$. Since $\Gamma$ is virtually torsion-free, $\Gamma \cap \RadLE(G)$ must be finite. So $\RadLE(G)$ is compact (and open). 
\end{proof}

\begin{Corollary} \label{cor-same-envelopes-minimax-EA}
	If $\Gamma$ is a minimax group, and $\Gamma$ and $\Lambda$ share a cocompact envelope, then $\Lambda$ lies in a short exact sequence $1 \to \Delta \to \Lambda \to Q \to 1$ such that $Q$ is virtually minimax, and $\Delta$ can be written as an increasing union $\Delta = \bigcup_n \Delta_n$ such that $\Delta_0$ is polycyclic and $\Delta_n$ is a finite index subgroup of $\Delta_{n+1}$ for all $n \geq 0$. 
\end{Corollary}

\begin{proof}
Let $G$ be a common cocompact envelope of $\Gamma$ and $\Lambda$. By Theorem  \ref{thm-gen-structure-envelopes} the subgroup $O = G^0 \, \RadLE(G)$ is an open normal subgroup of $G$. Since $\Gamma$ is cocompact in $G$, the discrete quotient $G/O$ is virtually minimax. Let $\Delta := \Lambda \cap O$. Note that $G$ is amenable, and therefore so is $O$. One can write $O$ as an increasing union $O = \bigcup_n O_n$  where each $O_n$ is open in $O$, and $O_n / G^0$ is compact. Note that $O_n$ necessarily has finite index in $O_{n+1}$. By Theorem \ref{thm-Hilbert-fifth} there is a compact normal subgroup $K_n$ of $O_n$ such that $K_n \cap \Lambda$ is trivial and $O_n/K_n$ is a virtually connected Lie group. Therefore for every $n$ the subgroup $\Delta_n := \Lambda \cap O_n = \Delta \cap O_n$ is isomorphic to a discrete and cocompact subgroup of the virtually connected amenable Lie group $O/K_n$. This implies that $\Delta_n$ is virtually polycyclic. Moreover $\Delta_n$ necessarily has finite index in $\Delta_{n+1}$, so the sequence $(\Delta_n)$ verifies the desired conclusion. 
\end{proof}

\subsection{The proof of Theorem  \ref{thm-intro-Rad-Lambda-fin}}

Recall that the FC-center of a group $\Delta$, denoted $\FC(\Delta)$, is the subgroup of $\Delta$ consisting of elements with a finite conjugacy class (equivalently, elements centralizing a finite index subgroup). 

\begin{Proposition} \label{prop-sequence-polycyclic}
	Suppose that a discrete group $\Delta = \bigcup_n \Delta_n$ is the increasing union of subgroups $\Delta_n$ such that: \begin{itemize}
		\item $\Delta_0$ is polycyclic;
		\item $\Delta_n$ is a finite index subgroup of $\Delta_{n+1}$ for all $n \geq 0$;
		\item $\RadLF(\Delta)$ is finite.
	\end{itemize}	
	Then $\Delta$ is virtually solvable of finite rank.
\end{Proposition}

\begin{proof}
The quotient $\Delta / \RadLF(\Delta)$, verifies the same assumptions as $\Delta$. Moreover being virtually solvable of finite rank is stable under forming an extension by a finite normal subgroup. Hence it is enough to prove the statement when $\RadLF(\Delta)$ is trivial. Since  $\Delta_0$ has finite index in  $\Delta_n$ for all $n$, $\Delta_0$ is a commensurated subgroup of $\Delta$. Hence there is a homomorphism $c_{\Delta, \Delta_0} : \Delta \to \mathrm{Comm}(\Delta_0)$, whose kernel $N$ is the set of elements of $\Delta$ that centralize a finite index subgroup of $\Delta_0$. We have $N \cap \Delta_n = \FC(\Delta_n)$ for all $n$ and $N = \bigcup \FC(\Delta_n)$. Note that $\FC(\Delta_n) = \Delta_n \cap \FC(\Delta_{n+1})$ and $\FC(\Delta_n)$ has finite index in $\FC(\Delta_{n+1})$ for all $n$. Let $L_n = \RadLF(\FC(\Delta_n))$. By Theorem 5.1 in \cite{Neumann-FC-51}, the subgroup $L_n$ coincides with the torsion elements of $\FC(\Delta_n)$, and $\FC(\Delta_n)/L_n$ is torsion-free abelian. In particular we have $L_n \leq L_{n+1}$ for all $n$. Therefore $L := \bigcup L_n$ is a locally finite subgroup of $\Delta$, which is also normal in $\Delta$. By assumption $\Delta$ has no non-trivial locally finite normal subgroup, we deduce that $L$ is trivial and hence $L_n$ is trivial for all $n$. So $\FC(\Delta_n)$ is torsion-free abelian. Moreover, being a subgroup of a virtually polycyclic group, $\FC(\Delta_n)$ is finitely generated. Since $\FC(\Delta_n)$ has finite index in $\FC(\Delta_{n+1})$, this shows that the abelian group $N = \bigcup \FC(\Delta_n)$ is of finite rank. 
	
	Since $\Delta_0$ is polycyclic, the group $\mathrm{Comm}(\Delta_0)$ is linear over $\Q$ \cite[Theorem 1.2]{Studenmund-comm}. Let $Q$ be the image of $c_{\Delta, \Delta_0} : \Delta \to \mathrm{Comm}(\Delta_0)$. Since $\Delta$ has no non-abelian free group, $Q$ also has this property. By the Tits' alternative, $Q$ is virtually solvable, and since $Q$ is linear over $\Q$, a finite index solvable subgroup of $Q$ is solvable of finite rank. Combined with the previous paragraph, this shows that $\Delta$ is virtually solvable of finite rank.
\end{proof}

For a group $\Lambda$, the assertions \enquote{$\Lambda$ has no normal subgroup that is infinite and locally finite} and \enquote{$\RadLF(\Lambda)$ is finite} are equivalent. Hence the following is an equivalent formulation of Theorem  \ref{thm-intro-Rad-Lambda-fin}.

\begin{Theorem} \label{thm-solv-f-rk}
Let	$\Gamma$ be a finitely generated solvable group of finite rank. Suppose that $\Gamma$ and $\Lambda$ share a cocompact envelope, and suppose that $\RadLF(\Lambda)$ is finite. Then $\Lambda$ is  virtually solvable of finite rank. 
\end{Theorem}

\begin{proof}	
Since $\Gamma$  finitely generated, $\Gamma$ is minimax (Proposition \ref{prop-background-minimax}). By Corollary \ref{cor-same-envelopes-minimax-EA}, there is a short exact sequence $ 1 \to \Delta \to \Lambda \to Q \to 1$ such that $Q$ is virtually minimax, and  $\Delta = \bigcup_n \Delta_n$ where $\Delta_0$ is polycyclic and $\Delta_n$ is a finite index subgroup of $\Delta_{n+1}$ for all $n \geq 0$. The subgroup $\Delta$ being normal in $\Lambda$, we have $\RadLF(\Delta ) \leq \RadLF(\Lambda)$. Hence $\RadLF(\Delta )$ is finite. Since $\Delta$ is the increasing union of the $\Delta_n$, Proposition \ref{prop-sequence-polycyclic} ensures $\Delta$ is virtually solvable of finite rank. Since $Q$ also has this property and being virtually solvable of finite rank is stable under extension, the proof is complete.
\end{proof}

\subsection{Solvable groups of finite rank of type $F_\infty$}

In this subsection we show that when $\Gamma$ is a solvable group of finite rank of type $F_\infty$, the assumption in Theorem  \ref{thm-solv-f-rk} that  the locally finite radical $\RadLF(\Lambda)$ of  $\Lambda$ is finite is not needed. 

 Let $\mathcal{C}_\infty$ denote the class of solvable groups of finite rank of type $F_\infty$. By a theorem of P.\ Kropholler, any solvable group of type $F_\infty$ is in $\mathcal{C}_\infty$ \cite{Kroph-FPinf-93} (and $\mathcal{C}_\infty$ also coincides with the class of so-called constructible groups). This result has been generalized by Kropholler--Martinez-P\'{e}rez--Nucinkis, who showed that any elementary amenable group of type $F_\infty$ is virtually in $\mathcal{C}_\infty$ \cite{Kro-Mart-Nuc-2009}. 

\begin{Theorem} \label{thm-finite-rank-F-infty}
	Let	$\Gamma$ be a group in $\mathcal{C}_\infty$ Suppose that $\Gamma$ and $\Lambda$ share a cocompact envelope. Then $\Lambda$ is virtually in $\mathcal{C}_\infty$. 
\end{Theorem}

\begin{proof}
It is a consequence of Corollary \ref{cor-same-envelopes-minimax-EA} that the group $\Lambda$ is elementary amenable. The group $\Lambda$ is also of type $F_\infty$ since $\Lambda$ is QI to $\Gamma$ and being of type $F_\infty$ is a QI-invariant. Hence by the aforementioned result from \cite{Kro-Mart-Nuc-2009}, $\Lambda$ is virtually in $\mathcal{C}_\infty$.
\end{proof}

\section{Flexibility results} \label{sec-flexibility}

The goal of this section is to provide two distinct constructions that show that the class of solvable groups of finite rank is not CE-rigid. These constructions are given respectively in \S \ref{subsec-flexible-first} and \S \ref{subsec-flexible-second}, and lead respectively to Theorem \ref{thm-intro-not-QI-rigid} and Theorem \ref{thm-intro-high-finiteness}.

\subsection{Preliminaries on Diestel-Leader graphs} \label{subsec-prelim-DL-graphs}

Let $n \geq 2$. We denote by $T_{n}$ the regular tree in which each vertex has degree $n + 1$. Let $\xi$ be an end of $T_n$. We denote by $\focal(n)$ the stabilizer of $\xi$ in the group $\Aut(T_n)$ of automorphisms of $T_n$. Since  $\Aut(T_n)$ acts transitively on the set of ends of $T_n$, end stabilizers are all conjugate, and hence the isomorphism class of $\focal(n)$ does not depend on $\xi$. We will denote by $b : T_{n} \to \Z $ a Busemann function associated to $\xi$ (the choice of $b$ consists of a normalization; different Busemann functions differ by an integer constant). The Busemann character associated to $\xi$ is a continuous homomorphism $\pi: \focal(n) \to \Z$ (independent of the choice of $b$). The kernel of $\pi$, denoted $\elliptic(n)$, is the set of automorphisms $g$ of $T_n$ such that $b(gv) = b(v)$ for every $v \in T_n$. Equivalently, this means that there is a geodesic ray towards $\xi$ that is fixed pointwise by $g$. The homomorphism $\pi$ is surjective, and we have $\focal(n)  = \elliptic(n) \rtimes \Z$. The group $\focal(n)$ acts vertex transitively on $T_n$. 

Let $d \geq 2$, and $n_1, \ldots , n_{d} \geq 2$. For each $i \leq d$, we fix and end $\xi_i$ of $T_{n_i}$ and a Busemann function $b_i : T_{n_i} \to \Z $  associated to $\xi_i$. 

\begin{Definition}
	The Diestel-Leader graph $\DL(n_1, \ldots,n_{d})$ is \[ \DL(n_1, \ldots,n_{d}) =  \left\lbrace (x_1, \ldots, x_{d}) \in T_{n_1} \times \cdots \times T_{n_{d}} \, | \, b_1(x_1) + \cdots + b_d(x_{d}) = 0   \right\rbrace. \] The incidence relation is defined by $(x_1, \ldots, x_d) \sim (x_1', \ldots, x_{d}')$ if and only if there are $i \neq j$ such that $x_r = x_r'$ for every $r \neq i,j$ and  $x_i \sim x_i'$ and $x_j \sim x_j'$. 
\end{Definition}

The graph $\DL(n_1, \ldots,n_{d})$ does not depend on the $\xi_i$'s or on the $b_i$'s up to isomorphism.

\begin{Notation}
	When $n_1= \ldots = n_{d} = n$, for simplicity we write $\DL_d(n)$ for $\DL(n, \ldots,n)$. 
\end{Notation}

\begin{Definition}
	The horocyclic product of $\focal(n_1), \ldots, \focal(n_{d})$ is \[ \bowtie_{i=1}^d \focal(n_i) =   \left\lbrace (g_1, \ldots, g_{d}) \in \focal(n_1) \times \cdots \times \focal(n_d) \, | \, \pi_1(g_1) + \cdots + \pi_d(g_{d}) = 0   \right\rbrace. \]
\end{Definition}

The group $\bowtie_{i=1}^d \focal(n_i)$ acts faithfully on $\DL(n_1, \ldots,n_{d})$ by $(g_1, \ldots, g_{d}) (x_1, \ldots, x_{d}) = (g_1 x_1, \ldots, g_d x_{d})$, and this action is transitive on vertices and proper. Theorem 2.7 in \cite{Bar-Neu-Woe} asserts that $\bowtie_{i=1}^d \focal(n_i)$ is precisely the group of isometries of $\DL(n_1, \ldots,n_{d})$ that do not permute non-trivially isometric factors. In particular $\bowtie_{i=1}^d \focal(n_i)$ is a finite index subgroup of the group of all isometries of $\DL(n_1, \ldots,n_{d})$.

\subsection{Actions on Diestel-Leader graphs}

The case $k=1$ in the following proposition is proven by Cornulier--Fisher-Kashyap in \cite{Cor-Fis-Kas}. The proof for $k \geq 2$  follows a similar scheme. It is interesting to note that our present setting involves two new conditions, \ref{item-ti-on-Rj} and \ref{item-ti-same-Rd} below, which do not appear in the case $k=1$. 

\begin{Proposition} \label{prop-action-higher-DL}
	Let $G$ be a locally compact group of the form $G = H \rtimes \Z^k$, $k \geq 1$. Suppose that there are generators $t_1, \ldots, t_k$ of the acting group $\Z^k$, and open subgroups $L_1, \ldots, L_{k+1}$ of $H$, and integers $n_1, \ldots,n_{k+1} \geq 2$ such that: \begin{enumerate}[label=(\arabic*)]
		\item \label{item-ti-on-Ri} for every $i=1, \ldots, k$, $L_i$ is contained in $t_i L_i t_i^{-1}$ as a subgroup of index $n_i$, and \[ \bigcup_{n \geq 1} t_i^n L_i t_i^{-n} = H;\]
		\item   \label{item-ti-on-Rk+1}  for every $i=1, \ldots, k$, $L_{k+1}$ is contained in $t_i^{-1} L_{k+1} t_i$ as a subgroup of index $n_{k+1}$, and \[  \bigcup_{n \geq 1} t_i^{-n} L_{k+1} t_i^{n} = H;\]
		\item  \label{item-ti-on-Rj} $t_i L_j t_i^{-1} = L_j$ for every  for every $i=1, \ldots, k$ and every $j \neq i,k+1$;
		\item \label{item-ti-same-Rd} $t_1^{-1} L_{k+1} t_1 = \cdots = t_k^{-1} L_{k+1} t_k$;
	\end{enumerate}
	Then there is a  continuous homomorphism $G \to \Isom(\DL(n_1, \ldots,n_{k+1}))$, and the associated action of $G$ on  $\DL(n_1, \ldots,n_{k+1})$ is proper if and only if:	\begin{enumerate}[resume, label=(\arabic*)]	\item  \label{item-action-DL-proper}   $\bigcap_{j=1}^{k+1} L_j$ is compact; 	\end{enumerate}	and cocompact  if and only if:
	\begin{enumerate}[resume, label=(\arabic*)]	\item \label{item-double-coset-fin} for every $\ell=1, \ldots, k+1$, the double coset space $ (\bigcap_{j \neq \ell} L_j) \backslash H / L_\ell$ is finite.\end{enumerate}	
\end{Proposition}

\begin{proof}
	Fix $i=1, \ldots, k$. Denote by $A_i$ the subgroup generated by the elements $(t_j)$ where $j \neq i$. By assumption \ref{item-ti-on-Rj} the subgroup $A_i$ normalizes $L_i$. Set $S_i := L_i \rtimes A_i$. Since $t_i$ centralizes $A_i$, we have $t_i S_i t_i^{-1} = (t_i L_i t_i^{-1}) \rtimes A_i$. In view of  \ref{item-ti-on-Ri} we have that $S_i$ is contained in $t_i S_i t_i^{-1}$ as a subgroup of index $n_i$, and  $(\bigcup_{n \geq 1} t_i^n S_i t_i^{-n} ) \rtimes \langle t_i \rangle = H \rtimes \Z^k = G$. Hence $G$ is the ascending HNN-extension over its subgroup $S_i$ associated to the isomorphism between $S_i$ and $t_i S_i t_i^{-1}$ performed the conjugation by $t_i$. This provides a vertex transitive action of $G$ on the tree $T_{n_i}$, the vertex set of $T_{n_i}$ being identified with the set of cosets $G / S_i$. The end associated to the geodesic ray $(t_i^{-n} S_i t_i)_{n \geq 1}$ is the unique end of $T_{n_i}$ fixed by $G$. We denote by $\rho_i: G \to \focal(n_i)$ the associated homomorphism. We also denote by $o_i$ the vertex corresponding to the coset $S_i$, and by $b_i$  the Busemann function normalized so that $b_i(o_i) = 0$. 
	
	Now let us denote by $B$ the subgroup  generated by the elements $(t_1 t_2^{-1}, \ldots, t_{k-1} t_k^{-1})$. By assumption \ref{item-ti-same-Rd}  the subgroup $B$ normalizes $L_{k+1}$. Set $S_{k+1}:= L_{k+1} \rtimes B$. Similarly as before we have $t_1 S_{k+1} t_1^{-1} = (t_1 L_{k+1} t_1^{-1}) \rtimes B$, and by \ref{item-ti-on-Rk+1} we have that $S_{k+1}$ contains $t_1 S_{k+1} t_1^{-1}$ as a subgroup of index $n_{k+1}$, and $(\bigcup_{n \geq 1} t_1^{-n}  S_{k+1} t_1^{n} ) \rtimes \langle t_1 \rangle = H \rtimes \Z^k = G$. So we also have a decomposition of $G$ as an ascending HNN-extension over its subgroup $S_{k+1}$, and hence an action of $G$  on the tree $T_{n_{k+1}}$. Let $\rho_{k+1} : G \to \focal(n_{k+1})$ the associated homomorphism. We denote by $o_{k+1}$ the vertex corresponding to the coset $S_{k+1}$, and by $b_{k+1}$ the Busemann function associated to the $G$-fixed end such that $b_{k+1}(o_{k+1}) = 0$.  
	
	We consider the diagonal action of $G$ on the product $T_{n_1} \times \cdots \times T_{n_k} \times T_{n_{k+1}}$ . Every element of $H$ acts on each tree as an elliptic element, and hence preserves $b_i$. For $i=1, \ldots, k$, $t_i$ leaves invariant $b_j$ for $j \neq i,k+1$, translates $b_i$ by $1$ and translates $b_{k+1}$ by $1$ in the opposite direction. So overall $G$ preserves $b_1 + \ldots + b_{k+1}$. Therefore the homomorphism $\rho := \rho_{1} \times \cdots \times \rho_{k} \times \rho_{k+1}: G \to \prod_{i=1}^{k+1} \focal(n_i)$ takes values in $\bowtie_{i=1}^{k+1} \focal(n_i)$, and we have an action on $G$ on $\DL(n_1, \ldots,n_{k+1})$.  The stabilizer of the vertex $o := (o_1, \ldots, o_{k+1})$ is the subgroup  $\bigcap_{j=1}^{k+1} S_j = \bigcap_{j=1}^{k+1} L_j$. This subgroup is open, so the action is continuous. And the action is proper if and only if $\bigcap_{j=1}^{k+1} L_j$ is compact. We have $\bowtie_{i=1}^{k+1} \focal(n_i) =  (\prod_{i=1}^{k+1} \elliptic(n_i)) \rho(G)$, so the number of orbits for the $G$-action on $\DL(n_1, \ldots,n_{k+1})$ is the same as the number of orbits for the $H$-action on $\prod_{i=1}^{k+1} b_i^{-1}(0)$. Up to an identification of $H/R_i$ with a subset of $G/S_i$, we have $b_i^{-1}(0) = H/R_i$. So we deduce that the $G$-action on $\DL(n_1, \ldots,n_{k+1})$ has finitely many orbits if and only if the $H$-action on $\prod_{i=1}^{k+1} H/R_i$ has finitely many orbits, which happens if and only if  \ref{item-double-coset-fin} holds.
\end{proof}

In the case where $K = \F_q(\!(t)\!)$ is a field of Laurent series, the following proposition is implicit in \cite{Bar-Neu-Woe}. Later we will use Proposition \ref{prop-metab-local-field-DLd} in the case $K = \Q_p$ (actually we will also use it indirectly in the case $K = \F_q(\!(t)\!)$ by appealing to a result from \cite{Bar-Neu-Woe}).

\begin{Proposition} \label{prop-metab-local-field-DLd}
	Let $d \geq 2$ and $K$ be a non-Archimedean local field of residue field of cardinality $q$. Let $\mathrm{Diag}_d^1(K)$ be the group of diagonal $(d \times d)$-matrices over $K$ of determinant of absolute value $1$. Then the group $G = K^d \rtimes \mathrm{Diag}_d^1(K)$ embeds as a closed cocompact subgroup in $\Isom(\DL_d(q))$. 
\end{Proposition}

\begin{proof}
	Let $R$ be the maximal compact subring of $K$. Let $\pi \in R$ such that $|\pi|$ generates the image of $|\cdot|: K^\times \to \R_{>0}$. The integer $q$ is the cardinality of $R / \pi R$. Let $\mathrm{Diag}_d(R)$ be the subgroup of $\mathrm{Diag}_d^1(K)$ of elements with coefficients in $R$. It is a compact open subgroup of  $\mathrm{Diag}_d^1(K)$. For $i = 1, \ldots, d-1$ we let $t_i \in \mathrm{Diag}_d^1(K)$ be the element with $\pi^{-1}$ at position $i$, with $\pi$ at position $d$, and $1$ elsewhere. The subgroup generated by $t_1, \ldots, t_{d-1}$ is $ \simeq \Z^{d-1}$, and one has $\mathrm{Diag}_d^1(K) = \mathrm{Diag}_d(R) \times \Z^{d-1}$. Therefore one has $G = H \rtimes \Z^{d-1}$ with $H =  K^d \rtimes \mathrm{Diag}_d(R)$.  We verify that the conditions of Proposition  \ref{prop-action-higher-DL} hold (with $k=d-1$). For $i=1, \ldots, d$, we set $L_i = (Ke_1 + \ldots + R e_i + \ldots + Ke_d  ) \rtimes \mathrm{Diag}_d(R)$. For $i = 1, \ldots, d-1$ and $n \geq 1$, we have $t_i^n L_i t_i^{-n} = (Ke_1 + \ldots + \pi^{-n} R e_i + \ldots + Ke_d  ) \rtimes \mathrm{Diag}_d(R)$ and $t_i^{-n} L_{d} t_i^{n} = (Ke_1 + \ldots + Ke_{d-1} +  \pi^{-n} R e_d ) \rtimes \mathrm{Diag}_d(R)$. So  \ref{item-ti-on-Ri}  and \ref{item-ti-on-Rk+1} hold with $n_1 = \cdots = n_d = q$.  \ref{item-ti-on-Rj}  and \ref{item-ti-same-Rd} are clear by definition of $t_i$. One has $\bigcap_{j=1}^{k+1} L_j = R^d \rtimes \mathrm{Diag}_d(R)$, which is compact, so \ref{item-action-DL-proper} holds. Finally for every $\ell=1, \ldots, d$ we have $ (\bigcap_{j \neq \ell} L_j) L_\ell = H$, so  \ref{item-double-coset-fin} holds.  And the action of $G$ on $\DL_d(q)$ is indeed faithful because the vertex stabilizer $R^d \rtimes \mathrm{Diag}_d(R)$ contains no non-trivial normal subgroup of $G$. 
\end{proof}

\subsection{First construction} \label{subsec-flexible-first}

Let $\mathbf{k}$ be either $\R$ or $\Q_p$. We say that $M \in \mathrm{GL}(2,\mathbf{k})$ is $\mathbf{k}$-bounded if the subgroup generated by $M$ is relatively compact in $\mathrm{GL}(2,\mathbf{k})$. Recall that for $M \in \mathrm{SL}(2,\R)$, $M$ is $\R$-bounded if and only if $|\mathrm{tr}(M)| < 2$. For $M \in \mathrm{SL}(2,\Q_p)$, we have the following classical characterization. We include a proof for the reader's convenience.

\begin{Lemma} \label{lem-eigenvalues-Qp}
	Let  $M \in \mathrm{SL}(2,\Q_p)$, and $p$ a prime. The following are equivalent:
	\begin{enumerate}[label=(\arabic*)]
		\item \label{item-M-Qp-bounded} $M$ is $\Q_p$-bounded;
		\item \label{item-M-eigenv} $M$ has either no eigenvalue in $\Q_p$, or two eigenvalues of absolute value $1$.
		\item \label{item-M-trace} $|\mathrm{tr}(M)|_p \leq 1$. 
	\end{enumerate}
\end{Lemma}

\begin{proof}
Let $\alpha = \mathrm{tr}(M)$. The characteristic polynomial is $f = t^2 - \alpha t +1$. If \ref{item-M-trace} holds then either $M = \pm I$, or $M$ is conjugate to $\left[ \left[0,-1 \right]; \left[1, \alpha \right]  \right] \in \mathrm{SL}(2,\Z_p)$. So \ref{item-M-Qp-bounded} holds. \ref{item-M-Qp-bounded} implies \ref{item-M-eigenv} is clear. Finally if \ref{item-M-trace} does not hold, i.e.\ if $|\alpha|_p > 1$, then Hensel's lemma implies that $f$ has a root in $\Q_p$ of absolute value $< 1$ (and consequently also has another root of absolute value $>1$). So \ref{item-M-eigenv} implies \ref{item-M-trace}. 
\end{proof}

Recall that a lattice $\Gamma$ is a product $G = G_1 \times \cdots \times G_r$ is irreducible if the projection of $\Gamma$ to any proper sub-product is non-discrete. When $r=2$ this  means $\Gamma$ has a non-discrete projection on each factor. 

\begin{Theorem} \label{thm-irr-R2-times-DL2}
	Let $M \in \mathrm{SL}(2,\Q)$. Suppose $M$ has infinite order and $M$ is $\R$-bounded. Write $\mathrm{tr}(M) = m/n$, with $m,n \in \Z$ relatively prime and $n \geq 2$. Let $A \leq \Q^2$ be the $\Z[M,M^{-1}]$-submodule of $\Q^2$ generated by $\Z^2$. Then the finitely generated group $\Gamma = A \rtimes_{M} \Z$  embeds as an irreducible cocompact lattice in $\mathrm{Isom}(\R^2) \times \Isom(\DL_2(n))$.
\end{Theorem}

\begin{proof}
	Since $M$ is $\R$-bounded, $|\mathrm{tr}(M)| < 2$. We note that since $M$ has infinite order, $ \mathrm{tr}(M) \notin \left\lbrace -1,0,1\right\rbrace$, so we indeed have $n \geq 2$ if we choose $n$ positive. Let $\pi$ be the set of prime divisors of $n$. By Lemma \ref{lem-eigenvalues-Qp} these are the primes for which $M$ is not $\Q_p$-bounded. Let $K \leq \mathrm{SL}(2,\R)$ be the closure of the subgroup generated by $M$. Since $M$ is $\R$-bounded, $K$ is a compact subgroup of $\mathrm{SL}(2,\R)$. Let $G_1 = \R^2 \rtimes K$, and let $G_2 = \left( \prod_{p \in \pi} \Q_p^2 \right) \rtimes \Z$, where the action of $\Z$ on each factor is by the matrix $M$. 
	
	\begin{Lemma} \label{lem-coco-embed}
		The map $i : \Gamma \to G_1 \times G_2$ defined by \[(x,y,M^n) \to (x,y,M^n) \times \left( (x,y), \ldots, (x,y),M^n\right)\] is an injective group homomorphism with discrete and cocompact image in $G_1 \times G_2$.
	\end{Lemma}
	
	\begin{proof}
		Set $V = \prod_{p \in \pi} \Q_p^2$. The image of $A$ in $\prod_{p \notin \pi} \Q_p^2$ is relatively compact (because $M \in \mathrm{SL}(2,\Z_p)$ for all but finitely many $p$ and $M$ is $\Q_p$-bounded for $p \notin \pi$). Hence $i(A)$  is discrete in $\R^2 \times V$. We shall check it is also cocompact. Let $W$ be the closure of the image of $A$ in $V$. The subgroup $W$ contains $\prod_{p \in \pi} \Z_p^2$, and $W$ is $M$-invariant. Since $M$ is not $\Q_p$ bounded for $p \in \pi$, this easily implies $W = V$. Since $A$ in addition contains $\Z^2$, it follows that $i(A)$ is indeed discrete cocompact in $\R^2 \times V$. 
		
		The subgroup $G_1 \times V$ is open in $G_1 \times G_2$, and $i(\Gamma) \cap (G_1 \times V) = i(A) \cap (\R^2 \times V)$. Discreteness of $i(\Gamma)$ in $G_1 \times G_2$  follows.  Since $i(\Gamma) \cdot (G_1 \times V) = G_1 \times G_2$, it also follows that $i(\Gamma)$ is cocompact in $G_1 \times G_2$. 
	\end{proof}

	The subgroup $K$ is conjugated to $\mathrm{SO}(2,\R)$ in $\mathrm{SL}(2,\R)$. So $G_1$ is conjugated to $\R^2 \rtimes  \mathrm{SO}(2,\R) = \mathrm{Isom}^+(\R^2)$ in $\R^2 \rtimes  \mathrm{SL}(2,\R)$.  Hence Lemma \ref{lem-coco-embed} implies that, in order to complete the proof of the proposition, it is enough to see that the group $G_2$ embeds as a closed and cocompact subgroup in $\Isom(\DL_2(n))$.  For, we rely on Proposition \ref{prop-action-higher-DL} (the case $k=1$, from \cite{Cor-Fis-Kas}). 
	
	For each $p \in \pi$, the matrix $M$ has two distinct eigenvalues in $\Q_p$, denoted $\lambda_{+}^{(p)}$ and $\lambda_{-}^{(p)}$, such that $|\lambda_{+}^{(p)}|_p < 1$ and $|\lambda_{-}^{(p)}|_p > 1$ (Lemma \ref{lem-eigenvalues-Qp}). We have  $|\lambda_{-}^{(p)}|_p = |\lambda_{+}^{(p)} + \lambda_{-}^{(p)}|_p = |\mathrm{tr}(M)| = p^{v_p(n)}$ and $|\lambda_{+}^{(p)}|_p = p^{-v_p(n)}$. Let $e_+^{(p)}, e_-^{(p)} \in \Q_p^2$ be eigenvectors of $M$ associated to $\lambda_{+}^{(p)}$ and $\lambda_{-}^{(p)}$. Set $L_1 = \prod_{p \in \pi}( \Q_p e_+^{(p)} + \Z_p e_-^{(p)} )$ and $L_2 = \prod_{p \in \pi}( \Z_p e_+^{(p)} + \Q_p e_-^{(p)})$. These are open subgroup of $V = \prod_{p \in \pi} \Q_p^2$. Moreover $V = L_1 + L_2$ because $\Q_p^2 = \Q_p e_+^{(p)} + \Q_p e_-^{(p)}$ for every $p \in \pi$, and $L_2 \cap L_1 = \prod_{p \in \pi}( \Z_p e_+^{(p)} + \Z_p e_-^{(p)})$ is  compact. One has \[ M (\Z_p e_+^{(p)} + \Q_p e_-^{(p)}) = \lambda_{+}^{(p)} \Z_p e_+^{(p)} + \lambda_{-}^{(p)} \Q_p e_-^{(p)} =  p^{v_p(n)} \Z_p e_+^{(p)} + \Q_p e_-^{(p)},  \] which is a subgroup of $\Z_p e_+^{(p)} + \Q_p e_-^{(p)}$ of index $p^{v_p(n)}$. Hence $M L_2$ is a subgroup of $L_2$ of index $\prod_{p \in \pi} p^{v_p(n)} = n$. Similarly $L_1$ is a subgroup of $ML_1$ of index $n$. Moreover if $t$ is the generator of $\Z$ in $G$ acting on $V$ via $M$, then one has \[ \bigcup_{n \geq 1} t^{-n} L_2 t^{n} = \bigcup_{n \geq 1}  t^{n} L_1  t^{-n} = V.\] Therefore conditions  \ref{item-ti-on-Ri}-\ref{item-ti-on-Rk+1}-\ref{item-action-DL-proper}-\ref{item-double-coset-fin}  of Proposition  \ref{prop-action-higher-DL} are satisfied (the other two conditions are void for $k=1$). The statement follows. 
\end{proof}

\begin{Remark}
The setting of Theorem  \ref{thm-irr-R2-times-DL2} shares some similarities with the one from \cite[Theorem 7.5]{Leary-Minasyan}. The difference between the two is that in our situation the second factor corresponds to a cocompact action on a Diestel--Leader graph, rather than a regular tree. In particular here the groups are amenable, while the ones in \cite{Leary-Minasyan} are never amenable. 
\end{Remark}

\begin{Theorem} \label{thm-SFR-share-product}
	For every $n \geq 2$, there is a finitely generated group of the form $\Gamma = \Z[1/n]^2 \rtimes \Z$ such that for every finite group $F$ of cardinality $n$, the groups $\Gamma$ and $\Lambda = \Z^2 \times F \wr \Z$ share the cocompact envelope $\mathrm{Isom}(\R^2) \times \Isom(\DL_2(n))$.
\end{Theorem}

\begin{proof}
	Take $M = \left[ \left[0,-1 \right] ,  \left[1,1/n \right] \right]$,  the companion matrix of $t^2 - t/n +1$.  We have $M \in \mathrm{SL}(2,\Q)$, $M$ has infinite order, and $M$ is $\R$-bounded. One verifies that the $\Z[M,M^{-1}]$-submodule of $\Q^2$ generated by $\Z^2$ is $\Z[1/n]^2$, so  Theorem  \ref{thm-irr-R2-times-DL2} implies that $G = \mathrm{Isom}(\R^2) \times \Isom(\DL_2(n))$ is a cocompact envelope of $\Gamma$. On the other hand, the group $\Isom(\DL_2(n))$ is a cocompact envelope of the wreath product $ F \wr \Z$ (see \cite{Bar-Neu-Woe,Cor-Fis-Kas} and historical references given there). Since the natural copy of $\Z^2$ in $\mathrm{Isom}(\R^2)$ is discrete and cocompact, the group $G$ is also a cocompact envelope of $\Lambda = \Z^2 \times F \wr \Z$.
\end{proof}

As recalled in the introduction, the wreath product $F \wr \Z$  is not virtually solvable provided $F$ is not solvable \cite{Erschler-instab}. Hence Theorem \ref{thm-intro-not-QI-rigid} from the introduction follows from Theorem \ref{thm-SFR-share-product}. 

\begin{Remark}
For every $d \geq 2$, one can also obtain similarly irreducible cocompact lattices in $\mathrm{Isom}(\R^d) \times \Isom(\DL_2(n))$.
\end{Remark}

\subsection{Second construction} \label{subsec-flexible-second}

We denote by $\mathbb{P}$  the set of prime numbers, and $\A$ the ring of adeles of $\Q$. Let $G$ be a linear algebraic group defined over $\Q$. We follow the notation from \cite{Borel-IHES-63}. Given a subring $R$ of a field extension of $\Q$, we denote by $G_R$ the group of elements of $G$ with coefficients in $R$ and determinant invertible in $R$. The groups $G_\R$ and $G_{\Q_p}$ are locally compact for the natural topologies, and $G_{\Z_p}$ is a compact open subgroup of $G_{\Q_p}$. We also denote $G_U = \prod_{p \in \mathbb{P}}  G_{\Z_p}$ and  $G_\A^{\infty} := G_\R \times G_U$, which are equipped with the product topology. The group $G_U$ is compact, and $G_\A^{\infty}$ is locally compact. The group of adeles of $G$, denoted $G_\A$, consists of elements $(g_\infty, (g_p)_{p \in \mathbb{P}} ) \in G_\R \times \prod_{p \in \mathbb{P}}   G_{\Q_p}$ such that $g_p \in G_{\Z_p}$ for all but finitely many $p$. The group $G_\A$ admits a locally compact group topology, for which the inclusion homomorphism $G_\A^{\infty} \to G_\A$ is continuous and has open image. 

We denote respectively by $j_p: G_{\Q} \to G_{\Q_p}$ and $j_\infty: G_{\Q} \to G_{\R}$ the natural homomorphisms. The homomorphism $j : G_{\Q}  \to G_\A$ defined by $j(g) = (j_\infty(g), (j_p(g))_\mathbb{P})$ is well-defined, and has discrete image. In the sequel when convenient we sometimes identify $G_{\Q}$ with its image in $ G_\A$.

The following is due to Ono and Borel: 

\begin{Theorem}[{\cite{Ono-Annals-1959, Borel-IHES-63}}]
	\label{thm-finite-doublecoset} 
	The double coset space $G_\A^{\infty} \backslash G_\A / G_{\Q}$ is finite. 
\end{Theorem}

The following is the main result of this section. 

\begin{Theorem} \label{thm-irr-Lie-DLd}
	For every $d \geq 2$, there is an infinite set of primes $S$ such that for every finite subset $\pi = \left\lbrace p_1, \ldots, p_k\right\rbrace $ of $S$,  there exists a group $\Gamma$ of the form \[ \Gamma = \Z[1/\pi]^d \rtimes \Z^{r}\] with $r = (k+1)(d-1)$, such that $\Gamma$ embeds as a cocompact irreducible lattice in the group \[ G= \R^d \rtimes (\R^\times)^{d-1} \times \Isom(\DL_d(p_1)) \times \cdots \times \Isom(\DL_d(p_k)).   \]
\end{Theorem}

\begin{proof}
For $ M \in \mathrm{SL}(d,\Z)$, we consider the centralizer $C(M)$ of $M$ in $\mathrm{SL}(d)$. It is algebraic and defined over $\Q$. In the sequel we  use notations introduced above. For a subring $R$ of a field extension $K / \Q$, $C(M)_R = \left\lbrace g \in \mathrm{SL}(d,R) \, | \, gM = Mg  \right\rbrace$.

The construction of groups $\Gamma$ as in the statement goes as follows. We start by choosing a matrix $M \in \mathrm{SL}(d,\Z)$ such that: \begin{enumerate}
	\item \label{item-M-diago-R} $M$ has $d$ distinct eigenvalues in $\R$;
	\item \label{item-M-centr} $C(M)_{\Z}$ has a finite index subgroup isomorphic to $\Z^{d-1}$. 
\end{enumerate} 

 Since $M$ has $d$ distinct eigenvalues in $\R$, $M$ is diagonalizable over $\R$ and $C(M)_{\R} \simeq (\R^\times)^{d-1}$. Since $C(M)_{\Z}$ embeds as a discrete subgroup of $C(M)_{\R}$ via $j_\infty$, we see that the assumption that $C(M)_{\Z}$ has a finite index subgroup isomorphic to $\Z^{d-1}$ is equivalent to saying that $j_\infty(C(M)_{\Z})$ is a cocompact subgroup of $C(M)_{\R}$. Matrices $M \in  \mathrm{SL}(d,\Z)$ with (\ref{item-M-diago-R}) and (\ref{item-M-centr}) exist for every $d \geq 2$ \cite[Theorems 1.14 and 2.7]{Prasad-Raghunathan}.

One has $j(C(M)_{\Q}) \cap C(M)_{\A}^\infty = j(C(M)_{\Z})$. Since $C(M)_U = \prod_{p \in \mathbb{P}}   C(M)_{\Z_p}$ is compact, and since $j_\infty(C(M)_{\Z})$ is a cocompact subgroup of $C(M)_{\R}$ by assumption, we have that $j(C(M)_{\Z})$ is a cocompact subgroup of $C(M)_{\A}^\infty$. Since the group $C(M)_{\A} $ is covered by right $j(C(M)_{\Q})$-translates of  finitely many cosets $C(M)_{\A}^\infty g_1, \ldots, C(M)_{\A}^\infty g_s$ by Theorem \ref{thm-finite-doublecoset}, we infer that \begin{equation}
\label{eq:cocompactness}
\text{the subgroup $j(C(M)_{\Q})$ is cocompact in $C(M)_{\A}$.} \tag{$\star$}
\end{equation}

Let $f \in \Z[t]$ be the minimal polynomial of $M$. Since the matrix $M$  has $d$ distinct eigenvalues in $\R$, the polynomial $f$ has degree $d$. Let $K / \Q$ be a field extension. Denoting $T(K)$ the  split torus of diagonal matrices in $\mathrm{SL}(d,K)$, the conditions \begin{enumerate}[label=(\alph*)]
	\item \label{item-M-diago} $M$ is diagonalizable over $K$; 
	\item \label{item-CM-diago} $C_M(K)$ is conjugate in $\mathrm{SL}(d,K)$ to $T(K)$; 
	%(which is isomorphic to $(K^\times)^{d-1}$);
	\item \label{item-f-splits} $f$ splits over $K$;
\end{enumerate}
are equivalent. The implications $\ref{item-CM-diago} \implies \ref{item-M-diago}  \implies \ref{item-f-splits}$ are clear. Since $\mathrm{char}(\Q) = 0$ and $f$ has $d$ distinct roots in $\R$, $f$ cannot have any multiple root in any field extension of $\Q$. Hence if \ref{item-f-splits} holds, then necessarily $f$ has $d$ distinct roots in $K$, and hence \ref{item-CM-diago} holds.

Let $L$ be the subfield of $\R$ generated by the roots of  $f$. Given a prime number $p$, we denote by $\Sigma_{L,p}$ the set of places of $L$ lying above $p$. For each $v \in  \Sigma_{L,p}$, the completion $L_v$ of $L$ with respect to $v$ is a finite extension of $\Q_p$. We are interested in the set of primes $p$ such that $f$ splits over $\Q_p$. These primes are exactly the ones such that $L_v = \Q_p$. Since $L / \Q$ is a Galois extension, the group $\mathrm{Gal}(L / \Q)$ acts transitively on $\Sigma_{L,p}$ \cite[Ch.\ I Proposition 9.1]{Neukirch-ANT}, so that $|\Sigma_{L,p}| = | \mathrm{Gal}(L / \Q) | / |\mathrm{Gal}(L / \Q)_v|$ where $\mathrm{Gal}(L / \Q)_v$ is the stabilizer of a fixed $v \in  \Sigma_{L,p}$. The group $\mathrm{Gal}(L / \Q)_v$ is isomorphic to $\mathrm{Gal}(L_v / \Q_p)$ \cite[Ch.\ II Proposition 9.6]{Neukirch-ANT}, so $|\mathrm{Gal}(L / \Q)_v| = \left[ L_v : \Q_p \right] $. Hence $|\Sigma_{L,p}| = \left[ L : \Q \right]  / \left[ L_v : \Q_p \right] $. If we denote $S_L = \left\lbrace p \,  : \,  |\Sigma_{L,p}| = \left[ L : \Q \right] \right\rbrace $, we have in particular $L_v = \Q_p$ if and only if $p \in S_L$. By Čebotarev theorem, the set of primes $S_L$ has density $1/\left[ L : \Q \right]$, and in particular $S_L$ is infinite \cite[Ch.\ VII Corollary 13.6]{Neukirch-ANT}.

Let $\pi = \left\lbrace p_i,\ldots,p_k \right\rbrace $ such that $\pi  \subset S_L$, and let  \[ \Gamma  =  \Z[1/\pi]^d \rtimes C(M)_{\Z[1/\pi]} \leq \Z[1/\pi]^d \rtimes \mathrm{SL}(d,\Z[1/\pi]).   \]

Denoting $i_p: \Q^d \to \Q_p^d$ and $i_\infty: \Q^d \to \R^d$ the canonical inclusions, we consider the homomorphism \begin{equation} \label{eq:prod-envelope} \varphi: \Gamma \to \R^d \rtimes C(M)_{\R} \times \prod_{p \in \pi} \Q_{p}^d \rtimes C(M)_{\Q_p}  \tag{$\star \star$} \end{equation} defined by \[ \varphi(x,g) = (i_\infty(x), j_\infty(g)), (i_{p_1}(x), j_{p_1}(g)), \ldots, (i_{p_k}(x), j_{p_k}(g)).  \]

Since $C(M)_{\R} \times \prod_{p \in \pi}  C(M)_{\Q_p} \times \prod_{p \notin \pi}  C(M)_{\Z_p} $ is an open subgroup in $C(M)_{\A}$, by (\ref{eq:cocompactness}) we have that $j(C(M)_{\Z[1/\pi]})$ is discrete and cocompact in $C(M)_{\R} \times \prod_{p \in \pi}  C(M)_{\Q_p} \times \prod_{p \notin \pi}  C(M)_{\Z_p} $. The subgroup  $\prod_{p \notin \pi}  C(M)_{\Z_p} $ being compact, we have that $\varphi(C(M)_{\Z[1/\pi]})$ is a discrete and cocompact subgroup of $C(M)_{\R} \times \prod_{p \in \pi} C(M)_{\Q_p}$. Since in addition $\varphi(\Z[1/\pi]^d)$ is a discrete and cocompact subgroup of $\R^d \times \prod_{p \in \pi} \Q_{p}^d$, it follows that $\varphi(\Gamma)$ is discrete and cocompact in $\R^d \rtimes C(M)_{\R} \times \prod_{p \in \pi} \Q_{p}^d \rtimes C(M)_{\Q_p}$. And $\varphi(\Gamma)$ has a non-discrete projection on each proper sub-product. 

For each $p \in \pi$, the minimal polynomial $f$ of $M$ splits over $\Q_p$ because $\pi  \subset S_L$. As recalled above, this means that $C(M)_{\Q_p}$ is conjugate in $\mathrm{SL}(d,\Q_{p})$ to the subgroup $T(\Q_p)$ of diagonal matrices in $\mathrm{SL}(d,\Q_{p})$. This conjugation induces an isomorphism $\Q_{p}^d \rtimes C(M)_{\Q_p} \to \Q_{p}^d \rtimes T(\Q_p)$. The group $T(\Q_p)$ is a closed and cocompact subgroup of the group  $\mathrm{Diag}_d^1(\Q_p)$ of diagonal $(d \times d)$-matrices over $\Q_p$ of determinant of absolute value $1$. Therefore by Proposition \ref{prop-metab-local-field-DLd} there is a continuous, injective homomorphism $\Q_{p}^d \rtimes T(\Q_p) \to  \Isom(\DL_d(p))$ with closed and cocompact image. Precomposing with the isomorphism $\Q_{p}^d \rtimes C(M)_{\Q_p} \to \Q_{p}^d \rtimes T(\Q_p)$, taking the product over $\pi$, and taking the  product with the real part $\R^d \rtimes C(M)_{\R}$,  which is isomorphic to $\R^d \rtimes (\R^\times)^{d-1}$, we indeed obtain an embedding of $\Gamma$ as a discrete cocompact subgroup of $ \R^d \rtimes (\R^\times)^{d-1} \times \prod_{p \in \pi}  \Isom(\DL_d(p))$. The group $C(M)_{\Z[1/\pi]}$ being embeddable as a discrete cocompact subgroup in \[ C(M)_{\R} \times \prod_{p \in \pi} C(M)_{\Q_p} \simeq (\R^\times)^{d-1} \times  \prod_{p \in \pi}  (\Q_p^\times)^{d-1}  \simeq (\R^\times)^{d-1} \times \Z^{k(d-1)} \times V  \] with $V$ compact, the rank $r$ of the torsion-free part of $C(M)_{\Z[1/\pi]}$ is equal to $r = d-1 + k(d-1) = (k+1)(d-1)$. Passing from $C(M)_{\Z[1/\pi]}$ to its finite index torsion-free part, we obtain a group of the form $\Z[1/\pi]^d \rtimes \Z^{r}$ satisfying all the requirements of Theorem \ref{thm-irr-Lie-DLd}. 
\end{proof}

See Example \ref{ex-explicit-fin-pres} below for an explicit example. 

\begin{Remark}
	Cornulier--Tessera gave in \cite[\S 8.1]{Cor-Tess-vanishing} a procedure that takes as input a  finitely generated virtually torsion-free solvable group of finite rank $\Gamma$, and produces a cocompact envelope  $G$ of $\Gamma$. The group $\Gamma$ is originally viewed as a subgroup of $\mathrm{GL}(n,\Q)$. The procedure starts by considering the nilpotent radical $N$ of $\Gamma$, and produces a cocompact envelope $\mathcal{N}$ of $N$ ($\mathcal{N}$ is defined as the product of the Zariski closure of the real projection and the closures of the $p$-adic projections). The envelope of $\Gamma$ is then defined as $G = \mathcal{N} \Gamma$. This $G$ is always nilpotent-by-(discrete virtually abelian) \cite[Proposition 8.4]{Cor-Tess-vanishing}. 
	
	The discrete and cocompact embedding $\varphi$ in (\ref{eq:prod-envelope}) is different from the one from  \cite{Cor-Tess-vanishing}. The map $\varphi$ can be divided in two independent parts: the first part consists in embedding the abelian normal subgroup in a product of local fields. That part is standard (and is also a special case of the above $\mathcal{N}$ from \cite{Cor-Tess-vanishing}). The second part consists in embedding the acting abelian subgroup discretely and cocompactly in a product, in a compatible way with respect to the embedding at the level of the abelian normal subgroup. The fact that it is possible to do so is very specific to the groups considered here. This is this second step that allows to obtain a global direct product decomposition in the envelope in (\ref{eq:prod-envelope}). 
\end{Remark}

\begin{Remark}
The case $d=2$ in Theorem \ref{thm-irr-Lie-DLd} and Theorem  \ref{thm-irr-R2-times-DL2} have in common that one tdlc factor in the product decomposition of the envelope is the isometry group of some $\DL_2(n)$. However these two settings remain quite different. First because Theorem \ref{thm-irr-Lie-DLd} allows more than one tdlc factor, but also mainly because the Lie group factors are different (the isometry group of $\R^2$ in Theorem  \ref{thm-irr-R2-times-DL2}, and the Lie group Sol (up to index two) in  Theorem \ref{thm-irr-Lie-DLd}). 
\end{Remark}

We now complete the proof of Theorem \ref{thm-intro-high-finiteness}. The group $\Gamma$ and the cocompact envelope $G =  \R^d \rtimes (\R^\times)^{d-1} \times \Isom(\DL_d(p))$ will be provided by Theorem \ref{thm-irr-Lie-DLd}, and the group $\Lambda$ will be taken to be a product of discrete and cocompact subgroups in each factor of $G$.

\begin{Definition}
Let $d \geq 2$ and $q$ a prime power such that $q \geq d-1$. Let \[ \Lambda_d(q) =  \F_q[t,t^{-1}, (t+1)^{-1}, \ldots, (t+d-2)^{-1}  ] \rtimes \Z^{d-1}, \] where the generators of $ \Z^{d-1}$ act by multiplication by $t, t+1, \ldots, t+d-2$.   
\end{Definition}

Here $\F_q$ is the finite field with $q$ elements. When $d=2$, $\Lambda_2(q)$ is the wreath product $\Z / q \Z \wr \Z$. When $d=3$,  $\Lambda_3(q) =  \F_q[t,t^{-1}, (t+1)^{-1} ] \rtimes \Z^{2}$ is Baumslag's metabelian group \cite{Baumslag-group} (more precisely, $\Lambda_3(q)$ is the positive characteristic analogue of the group from \cite{Baumslag-group}).

We make use of the following two results of Bartholdi--Neuhauser--Woess: 

\begin{enumerate}[label=(\Roman*)]
	\item \label{item-BNW-embed} $\Lambda_d(q)$ embeds as a discrete cocompact subgroup in $\Isom(\DL_d(q))$ \cite[Theorem 3.6]{Bar-Neu-Woe}.
	\item \label{item-BNW-type}  A discrete and cocompact subgroup of $\Isom(\DL_d(n))$ has type $F_{d-1}$ and not  $F_{d}$ \cite[Theorem 4.4-Corollary 4.5]{Bar-Neu-Woe}. 
\end{enumerate}

More specifically, one can verify that the embedding in \ref{item-BNW-embed} can be decomposed as an embedding of $\Lambda_d(q)$ as a discrete and cocompact subgroup in the group $\F_q(\!(t)\!)^d \rtimes \mathrm{Diag}_d^1(\F_q(\!(t)\!))$, followed by an embedding as a closed and cocompact subgroup of the latter in $\Isom(\DL_d(q))$ (Proposition \ref{prop-metab-local-field-DLd} for $K = \F_q(\!(t)\!)$). In view of \ref{item-BNW-embed}, statement \ref{item-BNW-type} applies notably to $\Lambda_d(q)$. The fact that $\Lambda_d(q)$ has type $F_{d-1}$ and not $F_{d}$ is also consequence of a result of Bux \cite[Corollary 3.5]{Bux}.

\begin{proof}[Proof of Theorem \ref{thm-intro-high-finiteness}]
We take $d \geq 2$. By Theorem \ref{thm-irr-Lie-DLd} (applied with $k=1$) one can find a prime $p$ larger than $d-1$, and a group $\Gamma$ of the form $\Gamma = \Z[1/p]^d \rtimes \Z^{2(d-1)}$ such that $\Gamma$ embeds as a cocompact irreducible lattice in $G =  \R^d \rtimes (\R^\times)^{d-1} \times \Isom(\DL_d(p)) = G_1 \times G_2$. The group $G_1$ admits a discrete and cocompact polycyclic subgroup $P$ (in the notation from the proof of Theorem \ref{thm-irr-Lie-DLd}, one can take $P \simeq \Z^d \rtimes C(M)_\Z$). By \ref{item-BNW-embed}  $G_2$ admits $\Lambda_d(q)$  as a discrete cocompact subgroup (since $p \geq d-1$). Hence $G$ is a cocompact envelope for $\Lambda = P \times \Lambda_d(q)$. By \ref{item-BNW-type} the group $\Lambda_d(q)$ has type $F_{d-1}$. Since $P$ is polycyclic, $P$ also has type $F_{d-1}$. Therefore $\Lambda$ as well. Since being of type $F_{d-1}$ is QI-invariant  \cite[Theorem 9.56]{Drutu-Kapovich}, $\Gamma$ has type $F_{d-1}$. The group $\Gamma$ is of finite rank and torsion-free, while $\Lambda$ is neither of finite rank nor virtually torsion-free, because $\Lambda_d(q)$ contains a subgroup isomorphic to an infinite dimensional vector space over $\F_p$. 
\end{proof}

\begin{Example} \label{ex-explicit-fin-pres}
	We consider an explicit example of a group $\Gamma$ as in Theorem \ref{thm-irr-Lie-DLd} with $d=3$ and $k=1$. The matrices $M, M_2, M_3, M_4$ and the claims below that are not seemingly obvious have been obtained with the assistance of the computer (Maple). Consider the polynomial $f(t) = t^3 - 5t^2 + 6t -1$. It has three roots $\alpha_1, \alpha_2, \alpha_3$ in $\R$, verifying $ 0 < \alpha_1 < 1 <  \alpha_2 < 2 < \alpha_3$. Let $M \in \mathrm{SL}(3,\Z)$ be the companion matrix of $f$. Consider the polynomial $h(t) = -f(1-t)$. It is the minimal polynomial of $M_2 := -M + I$.  It has constant coefficient equal to $-1$. Hence $M_2 \in \mathrm{SL}(3,\Z)$. Clearly $M_2$ commutes with $M$, so $M_2 \in C(M)_{\Z}$. Moreover $\left\langle M, M_2\right\rangle $ is isomorphic to $\Z^2$ (and hence $\left\langle M, M_2\right\rangle $ has finite index in $ C(M)_{\Z}$). Indeed, if it is not the case, there exist non-zero $m,n$ such that $M^n = M_2^m$. The equation $\alpha_1^n = (-\alpha_1 +1)^m$  then implies  $m,n$ have the same sign since $|\alpha_1|, |-\alpha_1 +1| < 1$, while $\alpha_2^n = (-\alpha_2 +1)^m$   implies  $m,n$ have opposite sign since $|\alpha_2| > 1$ and $|-\alpha_2 +1| < 1$, a contradiction.

	The smallest prime for which the polynomial $f$ splits over $\Q_p$ is $p=13$ (the next ones being $29$ and $41$). Consider the matrices $M_3 = 1/13 (M^2 + 5M + 4I)$ and $M_4 = 1/13 (2 M^2 + 3M + 4I)$. These matrices have determinant $1$, and hence belong to $C(M)_{\Z[1/13]}$. Let $(w_1,w_2,w_3)$ be a basis of $\Q_{13}^3$ consisting of eigenvectors of $M$. Consider the homomorphism $\varphi: C(M)_{\Z[1/13]} \to \Z^3$ which associates to $g \in C(M)_{\Z[1/13]}$ the vector $\varphi(g) = (v_{13}(\lambda_1(g)), v_{13}(\lambda_2(g)), v_{13}(\lambda_3(g)))$, where $v_{13}$ is the $13$-adic valuation and $\lambda_i(g)$ is the eigenvalue of $g$ associated to $w_i$. Since $M,M_2 \in \mathrm{SL}(3,\Z)$, the eigenvalues of $M$ and $M_2$ belong to $\Z_{13}^\times$. This means $M,M_2$ belong to the kernel of $\varphi$. Moreover for a suitable ordering of $(w_1,w_2,w_3)$ we have $\varphi(M_3) = (0,1,-1)$ and $\varphi(M_4) = (-1,-1,2)$. It follows that $\left\langle M_3, M_4\right\rangle $ is isomorphic to $\Z^2$ and that $\left\langle M, M_2, M_3, M_4\right\rangle $ is isomorphic to $\Z^4$, and is a finite index subgroup of $C(M)_{\Z[1/13]}$. By Theorem \ref{thm-irr-Lie-DLd} the associated semi-direct product $\Gamma = \Z[1/13]^3 \rtimes \Z^4$ sits as a discrete and cocompact subgroup in $\R^3 \rtimes (\R^\times)^{2} \times \Isom(\DL_3(13))$. Keeping the above notation, the group $\Lambda = P \times \Lambda_3(13)$, where $P = \Z^3 \rtimes_{(M, M_2)} \Z^2$ and $\Lambda_3(13) =  \F_{13}[t,t^{-1}, (t+1)^{-1} ] \rtimes \Z^{2}$, also embeds as a discrete and cocompact subgroup in $\R^3 \rtimes (\R^\times)^{2} \times \Isom(\DL_3(13))$. 
\end{Example}

\bibliographystyle{amsalpha}
\bibliography{commonenvelope}

\end{document}